\documentclass[12pt]{amsart}
\usepackage[margin=0.9in]{geometry}
\usepackage{amscd, amssymb, amsmath, wasysym}
\usepackage{graphicx}
\usepackage{amsfonts}
\usepackage{mathrsfs}    
\usepackage{mathtools}
\usepackage{eucal}     
\usepackage{latexsym}   
\usepackage{verbatim}   
\usepackage[all]{xy}   
\usepackage[dvipsnames]{xcolor}
\usepackage{bookmark}

\usepackage{hyperref}
 \hypersetup{
     colorlinks=true,
     linkcolor=NavyBlue,
     filecolor=NavyBlue,
     citecolor = TealBlue,
     urlcolor=magenta,
  }

\makeatletter

\newcounter{thmcounter}

\numberwithin{thmcounter}{section}
\numberwithin{equation}{section}

\newtheorem{theorem}[thmcounter]{Theorem}
\newtheorem{proposition}[thmcounter]{Proposition}
\newtheorem{lemma}[thmcounter]{Lemma}
\newtheorem{corollary}[thmcounter]{Corollary}

\theoremstyle{definition}

\newtheorem{remark}[thmcounter]{Remark}

\newtheoremstyle{claim}{9pt}{3pt}{}{\parindent}{\bf}{.}{1em}{}

\theoremstyle{claim}

\newenvironment{namelist}[1]{%
\begin{list}{}
{
\settowidth{\labelwidth}{#1}%
\setlength{\labelsep}{0.3em}%
\setlength{\leftmargin}{\labelwidth}%
\addtolength{\leftmargin}{\labelsep}}}{%
\end{list}}

\newcommand{\nZ}{\mathbf{Z}}                     % integer number
\newcommand{\nC}{\mathbf{C}}                     % complex number

\newcommand{\nP}{\mathbf{P}}                     % projective space

\newcommand{\sE}{\mathscr{E}}

\newcommand{\sO}{\mathscr{O}}                    % structure sheaf

\DeclareMathOperator{\Cliff}{Cliff}              % Cliff
\DeclareMathOperator{\coker}{coker}              % Coker
\DeclareMathOperator{\gon}{gon}                  % gon
\DeclareMathOperator{\uIm}{Im}                   % Im
\DeclareMathOperator{\Ker}{Ker}                  % Ker

\DeclareMathOperator{\pr}{pr} 

\DeclareMathOperator{\Supp}{Supp}                % Supp
\DeclareMathOperator{\Sym}{Sym}                  % Sym
\DeclareMathOperator{\rank}{rank}                % rank

\newcommand*{\longhookrightarrow}{\ensuremath{\lhook\joinrel\relbar\joinrel\rightarrow}}

\newcounter{rkcounter}             % set remark counter
\begin{document}

%======================Title Part ==================================
\title[Effective gonality theorem]{Effective gonality theorem on weight-one syzygies of algebraic curves}

\author{Wenbo Niu}
\address{Department of Mathematical Sciences, University of Arkansas, Fayetteville, AR 72701, USA}
\email{wenboniu@uark.edu}

\author{Jinhyung Park}
\address{Department of Mathematical Sciences, KAIST, 291 Daehak-ro, Yuseong-gu, Daejeon 34141, Republic of Korea}
\email{parkjh13@kaist.ac.kr}

\date{\today}

\subjclass[2020]{14Q20, 13A10}

\keywords{gonality conjecture, syzygies of an algebraic curve, symmetric product of an algebraic curve}

\begin{abstract} 
In 1986, Green--Lazarsfeld raised the gonality conjecture asserting that the gonality $\gon(C)$ of a smooth projective curve $C$ of genus $g\geq 2$ can be read off from weight-one syzygies of a sufficiently positive line bundle $L$ on $C$, and also proposed possible least degree of such a line bundle.  In 2015, Ein--Lazarsfeld proved the conjecture when $\deg L$ is sufficiently large, but the effective part of the conjecture remained widely open and was reformulated explicitly by Farkas--Kemeny. In this paper, we establish an effective vanishing theorem for weight-one syzygies, which implies that the gonality conjecture holds if $\deg L \geq 2g+\gon(C)$ or $\deg L = 2g+\gon(C)-1$ and $C$ is not a plane curve. As Castryck observed that the gonality conjecture may not hold for a plane curve when $\deg L = 2g+\gon(C)-1$, our theorem is the best possible and thus gives a complete answer to the gonality conjecture.
\end{abstract}

\maketitle

%======================Content Part ================================
\section{Introduction}

\noindent Throughout this paper, we work over the field $\nC$ of complex numbers. Let $C$ be a smooth projective curve of genus $g \geq 0$, and $B$ and $L$ be line bundles on $C$. Suppose that $L$ is globally generated, and write $S:=\Sym H^0(C, L)$ for the polynomial ring as the symmetric algebra of the vector space $H^0(C, L)$. The associated section module 
$$
R(C, B;L):=\bigoplus_{q\in \nZ}H^0(B\otimes L^q)
$$
is a finitely generated graded $S$-module and admits a minimal graded free resolution over $S$:
$$
\cdots\longrightarrow E_p\longrightarrow \cdots \longrightarrow E_2 \longrightarrow E_1\longrightarrow E_0\longrightarrow R(C, B;L)\longrightarrow 0.
$$
Each graded free $S$-module $E_p$ in the resolution has the form 
$$
E_p=\bigoplus_{q\in\nZ}K_{p,q}(C,B;L)\otimes_{\nC} S(-p-q),
$$
where $K_{p,q}(C,B;L)$ is the \emph{Koszul cohomology} defined as the cohomology at the middle of the Koszul-type complex
$$
\wedge^{p+1}H^0(L)\otimes H^0(B\otimes L^{q-1})\longrightarrow\wedge^pH^0(L)\otimes H^0(B\otimes L^q)\longrightarrow \wedge^{p-1}H^0(L)\otimes H^0(B\otimes L^{q+1}).
$$
We may regard $K_{p,q}(C, B; L)$ as the space of \emph{$p$-th syzygies of weight $q$}. For more details on Koszul cohomology, we refer the reader to the pioneering work \cite{Green:KoszulI, Green:Koszul2} of Green.
% on Koszul cohomology opens the door to the systematical study of the geometry of $C$ in $\nP^r$ by using Koszul cohomology $K_{p,q}(C;L)$. 

\medskip

One of  the most interesting cases is when $B=\sO_C$, in which case we simply write the associated section module as $R(C; L)$ and its Koszul cohomology groups as $K_{p,q}(C;L)$. It is elementary to see that $K_{i,0}(C; L) = 0$ if and only if $i \geq 1$, and if $L$ is nonspecial, i.e., $H^1(C, L)=0$, then $K_{i,j}(C; L) = 0$ for $j \geq 3$. Thus there are basically only two nontrivial weights for syzygies: $K_{i,1}(C; L)$ (weight-one) and $K_{i,2}(C; L)$ (weight-two). There has been a great deal of interests in the past decades to understand syzygies of algebraic curves with considerable effort on weight one and two especially. It turns out that they often reveal surprising connection between the intrinsic/extrinsic geometry and algebraic properties of the curves. Among others, generalizing classical results of Castelnuovo, Mattuck, Mumford, Saint-Donat, Fujita on equations defining algebraic curves, Green \cite[Theorem 4.a.1]{Green:KoszulI} proved his celebrated $(2g+1+p)$-theorem: \emph{if $\deg L \geq 2g+1+p$, then $K_{i,2}(C; L)=0$ for $0 \leq i \leq p$.} Green--Lazarsfeld  \cite{Lazarsfeld:SomeResults} classified the cases of $K_{p+1,2}(C;L) \neq 0$ when $\deg L = 2g+1+p$.
Geometrically, as $L$ is very ample, it gives an embedding
$$
C \subseteq \nP H^0(C, L) = \nP^r,~~\text{ where }~~r=r(L):=h^0(C, L)-1,
$$
and Green's theorem in particular means that $C \subseteq \nP^r$ is projectively normal when $\deg L \geq 2g+1$ and the defining ideal $I_{C|\nP^r}$ is generated by quadrics when $\deg L \geq 2g+2$. 
\medskip

It is natural to further ask what one can expect for weight-one syzygies $K_{i,1}(C; L)$ in the setting of Green's theorem. To continue our discussion leaving out the trivial cases, we assume $g \geq 2$ since for $g=0$ or $1$  the minimal graded free resolution of $R(C; L)$ can be explicitly computed. Note that $K_{0,1}(C; L) = 0$ and the dimension of  $K_{1,1}(C; L)$ is the number of minimal quadric generators of $I_{C|\nP^r}$. For the rest part in weight-one syzygies, Green--Lazarsfeld's nonvanishing theorem \cite[Appendix]{Green:KoszulI} implies that
\emph{if $\deg L \geq 2g+\gon(C)-2$, then $K_{r(L)-\gon(C),1}(C; L) \neq 0$ and hence $K_{p, 1}(C; L) \neq 0$ for $1 \leq p \leq r(L)-\gon(C)$}.
Recall that $\gon(C)$ is the \emph{gonality} of the curve $C$ defined as the minimal degree of a branched covering $C \to \nP^1$. 
The important issue here is whether weight-one syzygies actually vanish beyond the range above, and Green--Lazarsfeld's gonality conjecture  \cite[Conjecture 3.7]{Lazarsfeld:ProjNormCurve}  predicted that this is exactly the case. Precisely, in 1986, Green--Lazarsfeld  conjectured that if $\deg L \gg 0$, then 
\begin{equation}\label{eq:gonconj}
K_{p,1}(C; L) = 0~~\text{ for $p \geq  r(L)-\gon(C)+1$}.
\end{equation}
The conjecture was first verified for the case of very small value of gonality  \cite{Green:KoszulI, Teixidor} or for a general curve of each gonality \cite{Aprodu:GonConj, AV}. In 2015, Ein--Lazarsfeld \cite{Ein:Gonality} made a breakthrough by fully proving the gonality conjecture when $\deg L \gg 0$. However, the last piece of the whole puzzle, the effective gonality problem seeking an effective sharp bound for $\deg L$ to satisfy (\ref{eq:gonconj}), has remained widely open since then. Indeed, Green--Lazarsfeld \cite[page 87]{Lazarsfeld:ProjNormCurve} originally proposed that $\deg L \geq 2g+\gon(C)-1$ would suffice for (\ref{eq:gonconj}). Farkas--Kemeny  \cite[page 3]{Farkas:LinSyzGon} explicitly conjectured that this would be the case when $g \geq 4$, and verified their conjecture for general curves of each gonality. However, Castryck \cite{Castryck:LowerBoundGon} pointed out that Farkas--Kemeny's effective gonality conjecture is too optimistic: if $C \subseteq \nP H^0(C H)=\nP^2$ is a plane curve and $L:=\omega_C \otimes H$, then $\deg L = 2g+\gon(C)-1$ but $K_{r(L)-\gon(C)+1, 1}(C; L) \neq 0$. Moreover, Farkas--Kemeny noticed that if $L=\omega_C(\xi)$, where $\xi$ is any effective divisor of degree $\gon(C)$, then $\deg L = 2g+\gon(C)-2$ but $K_{r(L)-\gon(C)+1, 1}(C; L) \neq 0$. In somewhat slightly different flavor, Rathmann \cite{Rathmann:EffBd} showed that the gonality conjecture holds when $\deg L \geq 4g-3$, and noticed an example of a plane quartic curve for which the degree bound cannot be improved to $4g-4$. This is actually the only exception for $\deg L =4g-4$ when $L$ is very ample, as discussed in details in \cite{DNP}.

\medskip

In this paper, we prove the following effective gonality theorem, which completely resolves the effective gonality problem.

\begin{theorem}\label{thm:effgon} 
Let $C$ be a smooth projective curve of genus $g\geq 2$, and $L$ be a line bundle on $C$.
If $\deg L \geq 2g+\gon(C)-2$, then 
$$
K_{p,1}(C; L) \neq 0 ~~\Longleftrightarrow~~ 1 \leq p \leq r(L)-\gon(C)
$$
except for the following two cases:
\begin{enumerate}
\item [(1)] $C \subseteq  \nP H^0(C, H) = \nP^2$ is a plane curve of degree $\geq 4$, and $L=\omega_C \otimes H$. In this case, $\deg L = 2g+\gon(C)-1$.
\item [(2)] $C$ is arbitrary, and $L=\omega_C(\xi)$ for an effective divisor $\xi$ of degree $\gon(C)$ with $\dim |\xi|=1$. In this case, $\deg L = 2g+\gon(C)-2$.
\end{enumerate}	
In the exceptional cases, we have
$$
K_{p,1}(C; L) \neq 0 ~~\Longleftrightarrow~~ 1 \leq p \leq r(L)-\gon(C)+1.
$$	
\end{theorem}

The theorem unifies and clarifies all the previously known results on the gonality conjecture. It includes all aforementioned exceptional examples observed by Castryk and Farkas--Kemeny, and says that the desired vanishing (\ref{eq:gonconj}) of $K_{p,1}(C;L)$ holds without any exceptional cases as soon as $\deg L \geq 2g+\gon(C)$. It is interesting to point out that Castryck's example is the only exception for Green--Lazarsfeld's original expectation and Farkas--Kemeny's effective gonality conjecture, and Farkas--Kemeny's example is the only exception for the gonality conjecture when $\deg L = 2g+\gon(C)-2$. 

\medskip

As remarked by Green--Lazarsfeld \cite[page 87]{Lazarsfeld:ProjNormCurve}, one certainly wants a bound on the degree of $L$ independent of $\gon(C)$ for the gonality conjecture. Since $g \geq \gon(C)$, Theorem \ref{thm:effgon} can be applied to the case when $\deg L \geq 3g-2$.

\begin{corollary}\label{main:02} 
Let $C$ be a smooth projective curve of genus $g\geq 2$, and $L$ be a line bundle on $C$. If $\deg L \geq 3g-2$, then
$$
K_{p,1}(C; L) \neq 0 ~~\Longleftrightarrow~~ 1 \leq p \leq r(L)-\gon(C)
$$
except for the following two cases:
\begin{enumerate}
\item [(1)] $C \subseteq  \nP H^0(C, \omega_C) = \nP^2$ is a plane quartic curve and $L = \omega_C^2$ or $L=\omega_C^2(-x)$ for a point $x \in C$. In this case, $g=3$, and $\deg L = 3g-1=8$ or $\deg L = 3g-2=7$.
\item [(2)] $C$ is a curve of genus $g=2$ and $L=\omega_C^2$. In this case, $\deg L = 3g-2 = 4$.
\end{enumerate}
\end{corollary}

One can compute the whole Betti tables of the exceptional cases in Corollary \ref{main:02}. Notice that there is no exception when $\deg L \geq 3g$. Thus one can read off $\gon(C)$ from the minimal graded free resolution of $R(C;L)$ for one line bundle $L$ of degree $3g$. This was also expected by Green--Lazarsfeld \cite[page 87]{Lazarsfeld:ProjNormCurve}. On the other hand, Green's $(2g+1+p)$-theorem is sharp: Green--Lazarsfeld \cite[Theorem 1.2]{Lazarsfeld:SomeResults} showed that $K_{p+1,2}(C; L) \neq 0$ when $\deg L = 2g+1+p$ and $H^0(C, L \otimes \omega_C^{-1}) \neq 0$. In other words,
$$
K_{p,2}(C; L) \neq 0~~\Longleftrightarrow~~r(L)-g \leq p \leq r(L)-1.
$$
The required conditions automatically hold as soon as $\deg L \geq 3g-2$. Thus our result completes determining the vanishing and nonvanishing of all Koszul cohomology $K_{p,q}(C; L)$ and thereby the overall shape of the minimal graded free resolution of $R(C; L)$ when $\deg L \geq 3g-2$. 

\medskip

Theorem \ref{thm:effgon} follows from a more general effective vanishing theorem for weight-one syzygies of the module $R(C,B;L)$. Indeed, by the duality theorem (cf. \cite[Theorem 2.c.1]{Green:KoszulI}), we have
$$
K_{i,1}(C; L) = K_{r(L)-i-1, 1}(C, \omega_C; L)^{\vee}.
$$
It is well-known that $\gon(C)-2 \geq p$ if and only if $\omega_C$ is $p$-very ample. Thus the gonality conjecture can be restated as that  \emph{if $\omega_C$ is $p$-very ample, then $K_{p,1}(C, \omega_C; L)=0$}.
Recall that a line bundle $B$ on $C$ is said to be {\em $p$-very ample} if the restriction map
$$
H^0(C, B)\longrightarrow H^0(\xi, B|_{\xi})
$$ 
 is surjective for every effective divisor $\xi$ of degree $p+1$ (in other words, $\xi$ imposes independent conditions on the global sections of $B$). From this perspective, to prove the gonality conjecture, Ein--Lazarsfeld \cite[Theorem B]{Ein:Gonality} actually established the vanishing $K_{p,1}(C,B;L)=0$ when $B$ is $p$-very ample and $\deg L$ is sufficiently large. Soon after, Rathmann \cite[Theorem 1.2] {Rathmann:EffBd} further showed that this vanishing holds if $h^1(C, L) = h^1(C, L \otimes B^{-1})=0$.  In this paper, we prove the following effective vanishing theorem for weight-one syzygies, which significantly improves those vanishing results and from which  the effective gonality theorem follows as a consequence. 

\begin{theorem}\label{main:03} 
Let $C$ be a smooth projective curve of genus $g\geq 0$, and $B$ and $L$ be line bundles on $C$. Assume that $B$ is $p$-very ample. If
\begin{equation}\label{eq:h^1(L-B)intro}
h^1(C,L\otimes B^{-1})\leq r(B)-p-1,
\end{equation}
then $L$ is $p$-very ample and $K_{p,1}(C,B;L)=0$.
\end{theorem}

Our approach to prove the theorem is to conduct computation of the Koszul cohomology $K_{p,1}(C, B; L)$ on the symmetric product $C_{p+1}$ of the curve $C$. This idea was introduced by Voisin \cite{Voisin:GreenEven,Voisin:GreenOdd} and then extensively exploited by Ein--Lazarsfeld \cite{Ein:Gonality}, Rathmann \cite{Rathmann:EffBd}, and many others in the study of syzygies of algebraic curves \cite{Agostini:SecnatConj, DNP} and their secant varieties \cite{CKP, ENP}. Specifically, there is a tautological bundle $E_{p+1,B}$ on $C_{p+1}$ whose fiber over an effective divisor $\xi \in C_{p+1}$ of degree $p+1$ is $H^0(\xi, B|_{\xi})$ and whose global sections $H^0(C_{p+1}, E_{p+1, B}) = H^0(C, B)$. The $p$-very ampleness of $B$ gives rise to a surjective evaluation map 
$$
\operatorname{ev}_{p+1, B} \colon H^0(C, B) \otimes \sO_{C_{p+1}} \longrightarrow E_{p+1, B}.
$$
We can define the kernel bundle $M_{p+1,B}$ as the kernel of $\operatorname{ev}_{p+1, B}$. On the other hand, there is a line bundle $N_{p+1,L}=\det E_{p+1,L}$ on $C_{p+1}$ associated to $L$ with $H^0(C_{p+1}, N_{p+1, L})=\wedge^{p+1} H^0(C, L)$. The key observation made by Voisin (see e.g., \cite[Lemma 1.1]{Ein:Gonality}) is that $K_{p,1}(C,B;L)=0$ if and only if the multiplication map 
\begin{equation}\label{eq:multmapintro}
H^0(C, B) \otimes H^0(C_{p+1}, N_{p+1,L}) \longrightarrow H^0(C_{p+1}, E_{p+1,B} \otimes N_{p+1,L})
\end{equation}
is surjective. The surjectivity of the multiplication map follows from the cohomology vanishing
\begin{equation}\label{eq:H^1(MotimesN)}
H^1(C_{p+1}, M_{p+1,B} \otimes N_{p+1,L})=0.
\end{equation}
If $L$ is sufficiently positive as in the asymptotic situation treated by Ein--Lazarsfeld \cite{Ein:Gonality}, then $N_{p+1, L}$ is also sufficiently positive so that Fujita--Serre vanishing theorem immediately yields the desired cohomology vanishing (\ref{eq:H^1(MotimesN)}). However, making the vanishing effectively becomes much more subtle, and one really needs to develop a more sophisticated technique. 

\medskip

Our argument to prove Theorem \ref{main:03} is summarized as follows. 
First, we realize the dual to the above cohomology (\ref{eq:H^1(MotimesN)}) as a term 
$$
E_2^{p,0}=H^p(C_{p+1}, \wedge^{r(B)-p-1} M_{p+1, B} \otimes N_{p+1, \omega_C \otimes B \otimes L^{-1}})
$$
of the Leray spectral sequence for the projection $\pr_2 \colon C_{r(B)-p-1} \times C_{p+1} \to C_{p+1}$ associated to a line bundle $(N_{r(B)-p-1, B} \boxtimes N_{p+1,  \omega_C \otimes B \otimes L^{-1}})(-D_{r(B)-p-1, p+1})$, where
$$
D_{r(B)-p-1, p+1}:=\{(\xi, \eta) \in C_{r(B)-p-1} \times C_{p+1} \mid \Supp(\xi) \cap \Supp(\eta) \neq \emptyset\}
$$
is an effective divisor on $C_{r(B)-p-1} \times C_{p+1}$. If $E_2^{p,0}=E_{\infty}^{p,0}$ and $E^p=0$, then $E_2^{p,0}=0$. The required conditions can be deduced from the cohomology vanishing of line bundles
$$
H^i(C_j \times C_{p+1}, (N_{j, B} \boxtimes N_{p+1,  \omega_C \otimes B \otimes L^{-1}})(-D_{j, p+1}))=0
$$
for $0 \leq j \leq r(B)-p-1$ and $0 \leq i \leq p$. Next, we realize the dual to this cohomology as the converging term 
$$
E^{j+p+1-i} = H^{j+p+1-i}(C_j \times C_{p+1}, (S_{j, \omega_C \otimes B^{-1}} \boxtimes S_{p+1, L \otimes B^{-1}})(D_{j, p+1}))
$$
of the Leray spectral sequence for the projection $\pr_1 \colon C_j \times C_{p+1} \to C_j$. A pleasant geometric consequence of the condition (\ref{eq:h^1(L-B)intro}) is that
$$
\dim \underbrace{\{\xi \in C_j \mid h^1(C, (L \otimes B^{-1})(\xi)) \geq \ell \}}_{=:\mathcal{L}_j^{\ell}(\omega_C \otimes B \otimes L^{-1}) } \leq p+j-i-\ell~~\text{ for $\ell \geq 1$}.
$$
As $\Supp \big(R^\ell \pr_{1,*} (S_{j, \omega_C \otimes B^{-1}} \boxtimes S_{p+1, L \otimes B^{-1}})(D_{j, p+1}) \big) \subseteq \mathcal{L}_j^{\ell}(\omega_C \otimes B \otimes L^{-1})$, we get $E_2^{j+p+1-i-\ell, \ell}=0$ and hence $E^{j+p+1-i}=0$.

\medskip

Bearing in mind the multiplication map (\ref{eq:multmapintro}), one may think of Theorem \ref{main:03} as a far-reaching generalization of Green's $H^0$-Lemma (\cite[Theorem 4.e.1]{Green:KoszulI}), which is just $p=0$ case of our theorem. As a consequence of $H^0$-Lemma, Green \cite[Corollary 4.e.4]{Green:KoszulI} obtained that if 
$$
\deg L \geq 2g+1~~\text{ and }~~\deg B + \deg L \geq 4g+2,
$$
then $K_{0,1}(C, B; L)=0$. Based on the method of Green--Lazarsfeld \cite{Lazarsfeld:ProjNormCurve} for projective normality of algebraic curves, Butler \cite[Theorem 2]{Butler} and Pareschi \cite[Theorem 3 in Appendix]{Pareschi} significantly improve Green's result: roughly speaking, under some additional hypothesis on $B$ and $L$, if 
$$
\deg B + \deg L \geq 4g+1-2h^1(C,B) - 2h^1(C, L) - \Cliff(C),
$$
then $K_{0,1}(C, B; L)=0$. Here $\Cliff(C)$ is the \emph{Clifford index} of $C$ defined by the minimum of 
$$
\Cliff(A):=\deg A - 2h^0(C, A)+2
$$
for a line bundle $A$ on $C$ with $h^0(C, A) \geq 2$ and $h^1(C, A) \geq 2$ (if $g=2$ or $g=3$, there is no such a line bundle $A$, so we set $\Cliff(C):=\Cliff(\omega_C)=0$).
In this spirit, we show the following theorem for vanishing and nonvanishing of weight-one syzygies using Theorem \ref{main:03}.

\begin{theorem}\label{main:04}
Let $C$ be a smooth projective curve of genus $g\geq 2$, and $B$ and $L$ be line bundles on $C$. Assume that $B$ is $p$-very ample and $L$ is globally generated. If
$$
\deg L \geq 2g+p+1-h^1(C, B) ~~\text{ and }~~ \deg B + \deg L \geq 4g+2p-2h^1(C,B) - \Cliff(C),
$$
then $K_{i,1}(C, B; L) = 0$ for $0 \leq i \leq p-1$ and $K_{p,1}(C, B; L)=0$ except for the following cases: 
\begin{enumerate}
\item[(1)] $|B|$ is a base point free pencil $(p=0)$ and $h^1(C, L \otimes B^{-1}) = 1$. In this case, $K_{0,1}(C, B; L) = H^1(C, L \otimes B^{-1}) \neq 0$.
\item[(2)] $C \subseteq \nP H^0(C, B)=\nP^2$ is a plane curve of degree $\geq 4$ $(p=1)$ and $h^1(C, L \otimes B^{-1}) = 1$. In this case, $K_{0,1}(C, B; L) = 0$, and  $K_{1,1}(C, B; L) \neq 0$.
\item[(3)] $\deg L = 2g+p+h^0(C, L \otimes B^{-1})-h^1(C, B)-1$, $\deg B + \deg L = 4g+2p-2h^1(C,B) - \Cliff(C)$, and $L \otimes B^{-1}$ computes $\Cliff(C)$. If $B=\omega_C$, then $\gon(C)=p+2$ and either $C \subseteq \nP H^0(C, H)$ is a plane curve of degree $p+3 \geq 5$ and $L=\omega_C \otimes H$ $(\deg L = 2g+p+1)$ or $C$ is arbitrary with $g \geq 4$ and $L=\omega_C(\xi)$ $(\deg L = 2g+p)$ for an effective divisor $\xi$ of degree $p+2$ with $\dim |\xi|=1$. In this case, $K_{p,1}(C, B; L) \neq 0$.
\end{enumerate}
If furthermore $B$ is not $(p+1)$-very ample and $\deg L \geq 2g+p+1$, then $K_{p+1, 1}(C, B; L) \neq 0$.
\end{theorem}

This theorem almost contains the effective gonality theorem (Theorem \ref{thm:effgon}) as a special case when $B=\omega_C$ and fully recovers Green's $(2g+1+p)$-theorem when $B=L$ as $K_{p,2}(C;L)=K_{p,1}(C, L;L)$. Thus this somewhat technical-looking theorem provides a unified and strengthened statement on syzygies of algebraic curves. Moreover, a conjectural variant -- \emph{e.g., if $B$ is $p$-very ample and $L$ is $(p+1)$-very ample with $\deg L \geq 2g+p+1-2h^1(C, B) - \Cliff(C)$, and
$$
 \deg B + \deg L \geq 4g+2p+2-2h^1(C,B) - 2h^1(C, L) - 2\Cliff(C),
$$
then $K_{p,1}(C, B; L)=0$ except for a few cases} -- would solve  Green's conjecture \cite[Conjecture 5.1]{Green:KoszulI} and the Green--Lazarsfeld secant conjecture \cite[Conjecture 3.4]{Lazarsfeld:ProjNormCurve}.

\medskip

This paper is organized as follows. We begin in Section \ref{sec:prelim} with collecting relevant facts on symmetric products of algebraic curves. Section \ref{sec:vanishing} is devoted to the proof of our effective vanishing theorem for weight-one syzygies (Theorem \ref{main:03}). After discussing effective nonvanishing results for weight-one syzygies  in Section \ref{sec:nonvanishing}, we finally establish Theorem \ref{main:04} and the effective gonality theorem (Theorem \ref{thm:effgon}) in Section \ref{sec:effgonthm}.

\medskip

\noindent {\bf Acknowledgments:} We are grateful to Daniele Agostini, Lawrence Ein, and Sijong Kwak for valuable discussion and encouragement. We are especially indebted to Juergen Rathmann for sharing with us his preprints and  insights.

%======================New Section Starts================================
\section{Preliminaries on symmetric products of algebraic curves}\label{sec:prelim}

\noindent  In this section, we prepare the necessary techniques relating Koszul cohomology to vector bundles on symmetric products of algebraic curves. For more detailed discussion along this direction, we refer the reader to \cite{Ein:Gonality} and \cite{ENP}. 

\medskip

Let $C$ is a smooth projective curve of genus $g\geq 0$. For an integer $m \geq 0$, denote by $C_m$ the $m$-th symmetric product of $C$, which parameterizes effective divisors of degree $m$ on $C$. It is well-known that $C_m$ is a smooth projective variety of dimension $m$. Now, for integers $m,n \geq 1$, there is an addition map
$$
\sigma_{m,n} \colon C_{m}\times C_n \longrightarrow C_{m+n},~~(\xi,\xi') \longmapsto \xi+\xi',
$$
which is a finite morphism.
The incidence subvariety $D_{m,n}$ in the product $C_{m}\times C_n$ is an effective divisor defined to be the image of the  map
$$
j \colon C_{m-1}\times C_{n-1}\times C  \longrightarrow  C_{m}\times C_n, ~~(\xi, \xi', x) \longmapsto  (\xi+x,\xi'+x).
$$
As a set, $D_{m,n} = \{(\xi, \eta) \in C_{m} \times C_{n} \mid \Supp(\xi) \cap \Supp(\eta) \neq \emptyset\}$. 

\medskip

Consider the special case when $n=1$. In this case, $C_{m-1} \times C = D_{m,1}\subseteq C_{m}\times C$ is the universal family over $C_{m}$, as indicated in the following diagram 
$$
\xymatrix{
	D_{m,1}=C_{m-1} \times C \ \ar[dr]_{\sigma_{m-1,1}} \ar@{^{(}->}[r]^-j &C_{m}\times C\ar[d]^{\pr_{1}}\\
	&C_{m}
	}
$$
in which $\pr_1$ is the projection map. Given a line bundle $L$ on $C$, the {\em tautological bundle} associated to $L$ is a vector bundle of rank $m$ on $C_{m}$ defined by
$$
E_{m,L}:=\sigma_{m-1,1,*} (\sO_{C_{m-1}} \boxtimes L).
$$
Note that $H^0(C, E_{m,L})=H^0(C, L)$.
Next, we define a line bundle 
$$
N_{m,L}:=\det (E_{m,L}).
$$
There is a divisor $\delta_{m}$ on $C_{m}$ such that $\sO_{C_{m}}(-\delta_{m})=N_{m,\sO_C}$. Note that $\sigma_{m-1,1}^*\delta_{m} = D_{m-1,1} \subseteq C_{m-1} \times C$. Since $E_{m, \omega_C}$ is the cotangent bundle of $C_m$, we have $\omega_{C_m} = N_{m, \omega_C}$.
Now, consider a line bundle $S_{m, L}$ in $C_{m}$ which is the invariant descent of the $m$-fold box product   
$$
L^{\boxtimes m}:=\underbrace{L\boxtimes  \cdots \boxtimes L}_{\text{$m$ times}}
$$ 
from the $m$-th Cartesian product $C^{m}$ to the $m$-th symmetric product $C_{m}$ under the action of the symmetric group $\mathfrak{S}_{m}$. We have $N_{m, L} = S_{m, L}(-\delta_{m})$. The line bundles $N_{m,L}$ and $S_{m,L}$ are playing critical roles in the computation of Koszul cohomology on the symmetric products of $C$. Fortunately, the cohomology of these two line bundles are well understood as summarized in the following lemma. 

\begin{lemma}[{\cite[Lemma 2.4]{Agostini:SecnatConj}, \cite[Lemma 3.7]{ENP}}]\label{lem:H^i(N)H^i(S)}
For any $i \geq 0$, we have
$$
H^i(C_m, N_{m,L}) = \wedge^{m-i} H^0(C, L) \otimes S^i H^1(C, L)~\text{ and }~ H^i(C_m, S_{m,L}) = S^{m-i} H^0(C, L) \otimes \wedge^i H^1(C, L).
$$
\end{lemma}

Given a line bundle  $B$ on $C$, the fiber of $E_{p+1, B}$ over $\xi \in C_{p+1}$ is $E\otimes k(\xi)=H^0(\xi, B|_{\xi})$. Thus the restriction map $H^0(C, B)\rightarrow H^0(\xi, B|_{\xi})$ is globalized to an evaluation map 
$$
\operatorname{ev}_{p+1, B} \colon H^0(C, B)\otimes \sO_{C_{p+1}}\longrightarrow E_{p+1,B}.
$$ 
Note that $B$ is $p$-very ample if and only if $\operatorname{ev}_{p+1, B}$ is surjective if and only if $E_{p+1, B}$ is globally generated. Following Voisin \cite{Voisin:GreenEven, Voisin:GreenOdd} and Ein--Lazarsfeld \cite{Ein:Gonality}, we briefly explain how one can carry out the computation of the Koszul cohomology $K_{p,1}(C, B; L)$ on $C_{p+1}$. Formally, consider the Koszul-type complex
$$
\wedge^{p+1}H^0(C,L)\otimes H^0(C,B)\stackrel{d_{p+1}}{\longrightarrow}\wedge^pH^0(C,L)\otimes H^0(C,B\otimes L)\stackrel{d_p}{\longrightarrow} \wedge^{p-1}H^0(C,L)\otimes H^0(C,B\otimes L^{2}).
$$
The Koszul cohomology group $K_{p,1}(C,B;L)$ can be defined as the quotient $\Ker(d_{p})/\uIm(d_{p+1})$. Suppose now that $B$ is $p$-very ample so that the evaluation map $\operatorname{ev}_{p+1, B}$ is surjective. Let $M_{p+1, B}:=\ker(\operatorname{ev}_{p+1, B})$ be the kernel bundle, which fits into a short exact sequence
\begin{equation}\label{eq:sesforM_{p+1,B}}
0 \longrightarrow M_{p+1,B} \longrightarrow H^0(C, B)\otimes \sO_{C_{p+1}} \longrightarrow E_{p+1,B} \longrightarrow 0.
\end{equation}
The key observation due to Voisin (see e.g., \cite[Lemma 1.1]{Ein:Gonality}) is that the multiplication map $m_{p+1}$ in the exact sequence
$$
0\longrightarrow H^0(C_{p+1}, M_{p+1,B}\otimes N_{p+1,L}) \longrightarrow H^0(C, B)\otimes H^0(C_{p+1}, N_{p+1,L})\stackrel{m_{p+1}}{\longrightarrow} H^0(C_{p+1}, E_{p+1,B}\otimes N_{p+1,L})
$$
derived by taking cohomology of (\ref{eq:sesforM_{p+1,B}}) is exactly the Koszul differential map $d_{p+1}$ and the space $ H^0(C_{p+1}, E_{p+1,B}\otimes N_{p+1,L})$ is exactly the space $\Ker(d_p)$. 
Thus $K_{p,1}(C, B; L)$ is the cokernel of the multiplication map $m_{p+1}$, or equivalently one has  the following exact sequence 
$$0\longrightarrow K_{p,1}(C,B;L)\longrightarrow H^1(C_{p+1},  M_{p+1,B}\otimes N_{p+1,L})\longrightarrow H^0(C, B)\otimes H^1(C_{p+1}, N_{p+1,L}). $$
In particular,  if $H^1(C_{p+1}, N_{p+1, L})=0$ (e.g., this holds when $H^1(C, L)=0$ by Lemma \ref{lem:H^i(N)H^i(S)}), then 
$K_{p,1}(C, B; L) = H^1(C_{p+1}, M_{p+1, B}\otimes N_{p+1,L})$. 
In general, we obtain the following lemma, which will be used to prove our effective vanishing theorem for weight-one syzygies (Theorem \ref{main:03}). 

\begin{lemma} [{\cite[Lemma 1.1]{Ein:Gonality}}]\label{lem:H^1(MotimesN)=0=>K_{p,1}=0}
Assume that $B$ is a $p$-very ample line bundle and $L$ is a line bundle on $C$. If $H^1(C_{p+1},M_{p+1,B}\otimes N_{p+1,L})=0$, then $K_{p,1}(C,B;L)=0$.
\end{lemma}

The vanishing of $H^1(C_{p+1},M_{p+1,B}\otimes N_{p+1,L})$ implies the vanishing of $K_{p,1}(C,B;L)$ but the converse needs not be true.  Under suitable conditions, one may assume $H^1(C, L) = 0$ and hence $H^1(C_{p+1}, N_{p+1, L})=0$, so $H^1(C_{p+1},M_{p+1,B}\otimes N_{p+1,L})=0$ if and only if $K_{p,1}(C, B; L)=0$. However, roughly speaking, as it is hard to control the kernel bundle $M_{p+1,B}$, an effective result for $H^1(C_{p+1},M_{p+1,B}\otimes N_{p+1,L})=0$ is usually quite difficult to obtain directly. The next lemma will serve our purpose to manage the kernel bundle $M_{p+1,B}$ and overcome such difficulty. It provides a way to lift a vector bundle of the form $\wedge^{n-i} M_{p+1,B}$ on $C_{p+1}$ to a line bundle of the form $(\sO_{C_{p+1}} \boxtimes N_{n,B})(-D_{p+1,n})$ on $C_{p+1} \times C_n$, and therefore successfully enables us to transfer the vanishing of cohomology of vector bundle $\wedge^{n-i} M_{p+1,B} \otimes N_{p+1, L}$ to the vanishing of cohomology of line bundle  $(N_{p+1, L} \boxtimes N_{n,B})(-D_{p+1,n})$.

\begin{lemma}[{cf. \cite[Lemma 3.3]{CKP}}]\label{lem:R^ipr_1}
Assume that $B$ is a $p$-very ample line bundle on $C$. Let $\pr_1 \colon C_{p+1} \times C_n \to C_{p+1}$ be the projection map. 
Then there is a natural isomorphism
$$
R^i\pr_{1,*} \big((\sO_{C_{p+1}}\boxtimes N_{n,B})(-D_{p+1,n})\big)\cong\wedge^{n-i}M_{p+1,B}\otimes S^iH^1(C, B)~~\text{ for $i \geq 0$}.
$$
\end{lemma}

\begin{proof}
Notice that for $\xi \in C_{p+1}$, the restriction of the sheaf $(\sO_{C_{p+1}}\boxtimes N_{n,B})(-D_{p+1,n})$ to the fiber $\pr^{-1}_1(\xi) \cong C_n$ is $N_{n,B(-\xi)}$. By Lemma \ref{lem:H^i(N)H^i(S)},
$$
\begin{array}{rcl}
H^i(C_{n},N_{n,B(-\xi)})&=&\wedge^{n-i}H^0(C, B(-\xi))\otimes S^iH^1(C, B(-\xi)) \\
&\subseteq &\wedge^{n-i}H^0(C, B)\otimes S^iH^1(C, B) \,=\, H^i(C_n, N_{n, B}).
\end{array}
$$
As $B$ is $p$-very ample, both $h^0(C, B(-\xi))$ and $h^1(C, B(-\xi))$ are constants  independent on the choice of $\xi$ and in addition $H^1(C, B(-\xi))=H^1(C, B)$. By the base change, there is a natural isomorphism 
$$
u^i_{\xi} \colon R^i\pr_{1,*}\big( (\sO_{C_{p+1}}\boxtimes N_{n,B})(-D_{p+1,n})\big)\otimes \nC(\xi)\stackrel{\sim}{\longrightarrow}H^i(C_{n},N_{n,B(-\xi)}).
$$
Now, push down the inclusion map 
$$
\sO_{C_{p+1}}\boxtimes N_{n,B}(-D_{p+1,n})\longhookrightarrow \sO_{C_{p+1}}\boxtimes N_{n,B}
$$
by $\pr_1$ to yield 
$$
\delta^i \colon R^i \pr_{1,*}\big( (\sO_{C_{p+1}}\boxtimes N_{n,B})(-D_{p+1,n})\big) \longrightarrow \underbrace{H^i(C_{p+1}, N_{n,B})\otimes \sO_{C_{p+1}}}_{\mathclap{=\wedge^{n-i}H^0(C, B)\otimes S^iH^1(C, B)\otimes \sO_{C_{p+1}}}} ~~\text{for $i\geq 0$}.
$$
Tensoring $\delta^i$ with the residue field $\nC(\xi)$ for each point $\xi\in C_{p+1}$ induces a map 
$$
\delta^i_\xi \colon R^i \pr_{1,*}\big( (\sO_{C_{p+1}}\boxtimes N_{n,B})(-D_{p+1,n})\big) \otimes \nC(\xi) \longrightarrow \wedge^{n-i}H^0(C, B)\otimes S^iH^1(C, B) 
$$
which factors through the map $u^i_{\xi}$ and has the image $H^i(C_{n},N_{n,B(-\xi)})$.
On the other hand, the short exact sequence (\ref{eq:sesforM_{p+1,B}}) induces an inclusion 
$$
\alpha^i \colon \wedge^{n-i}M_{p+1,B}\otimes S^iH^1(C, B)\longhookrightarrow \wedge^{n-i}H^0(C, B)\otimes S^iH^1(C, B)\otimes \sO_{C_{p+1}}.
$$ 
We claim that the map $\delta^i$ factors through $\alpha^i$. Indeed, for $\xi\in C_{p+1}$, we have 
$$\wedge^{n-i}M_{p+1,B}\otimes S^iH^1(C, B)\otimes \nC(\xi)=\wedge^{n-i}H^0(C, B(-\xi))\otimes S^iH^1(C,, B)=H^i(C_{n},N_{n,B(-\xi)}).
$$
This means that 
$$
\delta^i_{\xi}\Big(R^i\pr_{1,*}\big((\sO_{C_{p+1}}\boxtimes N_{n,B})(-D_{p+1,n})\big)\otimes \nC(\xi)\Big)=\wedge^{n-i}M_{p+1,B}\otimes S^iH^1(C, B)\otimes \nC(\xi).
$$ 
Thus $\delta^i$ maps $R^i\pr_{1,*}\big( (\sO_{C_{p+1}}\boxtimes N_{n,B})(-D_{p+1,n})\big)$ surjectively onto $\wedge^{n-i}M_{p+1,B}\otimes S^iH^1(C, B)$, and it must be an isomorphism since both are vector bundles of the same rank. 
\end{proof}

%======================New Section Starts================================
\section{Effective vanishing for weight-one syzygies of algebraic curves}\label{sec:vanishing}

\noindent We continue to assume that $C$ is a smooth projective curve of genus $g \geq 0$ and $B,L$ are line bundles on $C$. This section is devoted to proving Theorem \ref{main:03}. For this purpose, we develop a general technique to show the vanishing of the cohomology groups of the form $H^m(C_{p+1},\wedge^n M_{p+1,B}\otimes N_{p+1, L})$. The basic idea is that such vanishing can be transferred to the vanishing of certain cohomology groups of the sheaves $N_{p+1,L}\boxtimes N_{p+1,B}(-D_{p+1,n})$ on $C_{p+1}\times C_{n}$. Then these vanishing are controlled  by the dimensions of suitable loci in the space $C_n$ and eventually can be controlled by the dimension of global sections of $L$.

\begin{proposition}\label{prop:H^m(MotimesE)=0}
Let $B$ be a $p$-very ample line bundle on $C$, and $\sE$ be a locally free sheaf on $C_{p+1}$. For integers $m,n\geq 0$, we have
	$$H^m(C_{p+1},\wedge^nM_{p+1,B}\otimes \sE)=0$$
provided that the following two conditions are satisfied:
\begin{enumerate}
	\item [(1)] $H^m(C_{p+1}\times C_n, (\sE\boxtimes N_{n,B})(-D_{p+1,n}))=0$.
	\item [(2)] $H^{m-j}(C_{p+1}\times C_{n-i}, (\sE\boxtimes N_{n-i,B})(-D_{p+1,n-i}))=0$ for all $i,j$ with $1\leq i\leq m-1$ and $i+1\leq j\leq 2i$.
\end{enumerate} 
\end{proposition}

\begin{proof} 
Write  $\pr_1 \colon C_{p+1}\times C_n\rightarrow C_{p+1}$ for the projection map. By Lemma \ref{lem:R^ipr_1}, pushing down the  sheaf $(\sE\boxtimes N_{n,B})(-D_{p+1,n})$ by $\pr_1$ yields
$$
R^j\pr_{1,*}\big((\sE\boxtimes N_{n,B})(-D_{p+1,n})\big)=S^jH^1(C,B)\otimes \wedge^{n-j}M_{p+1,B}\otimes \sE~~ \text{ for $j\geq 0$},
$$
and gives rise to a Leray spectral sequence
$$
E^{i,j}_2:=S^{j}H^1(C,B)\otimes H^{i}(C_{p+1},\wedge^{n-j}M_{p+1,B}\otimes \sE)
$$
converging to
$$
E^{i+j}:=H^{i+j}(C_{p+1}\times C_n,\sE\boxtimes N_{n,B}(-D_{p+1,n})).
$$ 
To prove the proposition, we just need to show $E^{m,0}_2=0$.
From the convergence of the spectral sequence, we observe that the vanishing $E^{m,0}_2=0$ holds if  $E^m=0$ and $E^{m,0}_2=E^{m,0}_{\infty}$. The former condition $E^m=0$ is nothing but the one given in $(1)$, while  the latter condition $E^{m,0}_2=E^{m,0}_{\infty}$ is implied by the vanishing
$$
E^{m-2,1}_2=E^{m-3,2}_2=\cdots =E^{0,m-1}_2=0.
$$
Thus we reduce the problem to showing that
$$
H^{m-1-i}(C_{p+1}, \wedge^{n-i}M_{p+1,B}\otimes \sE)=0 ~~\text{ for $1\leq i\leq m-1$}.
$$
To this end, for each $i$, we argue exactly in the same way above by utilizing the Leray spectral sequence for pushing down of the sheaf $(\sE\boxtimes N_{n-i,B})(-D_{p+1,n-i})$ by the projection map $\pr_1 \colon C_{p+1} \times C_{n-i} \to C_{p+1}$. It turns out that the vanishing of $H^{m-1-i}(C_{p+1}, \wedge^{n-i}M_{p+1,B}\otimes \sE)$ can be reduced to showing the vanishing
$$
H^{m-2-k}(C_{p+1}, \wedge^{n-k} M_{p+1, B} \otimes \sE) = 0~~\text{ for $i+1 \leq k \leq m-2$}
$$
and the vanishing 
$$H^{m-1-i}(C_{p+1}\times C_{n-i},\sE\boxtimes N_{n-i,B}(-D_{p+1,n-i}))=0.$$
But the latter one is already given in condition $(2)$ for $j=i+1$. 
So if we continue to repeat the above arguments to handle vanishing for all the terms of the form $\wedge^* M_{p+1,B}\otimes \sE$ created in the process, eventually we are able to transfer all  required cohomology vanishing on $C_{p+1}$ to the ones exactly given in condition $(2)$.
Therefore, we conclude $H^m(C_{p+1},\wedge^nM_{p+1,B}\otimes \sE)=0$ as desired.
\end{proof}

\begin{remark} 
(1)  When $m=0$ or $1$ in Proposition \ref{prop:H^m(MotimesE)=0}, the second vanishing hypothesis becomes empty so that the vanishing condition $H^m(C_{p+1}\times C_n,\sE\boxtimes N_{n,B}(-D_{p+1,n}))=0$ alone implies the vanishing $H^m(C_{p+1},\wedge^nM_{p+1,B}\otimes \sE)=0$.\\[3pt]
(2) If $H^1(C,B)=0$, then the term $E^{i,j}_2$ of  the Leray spectral sequence for the projection map $\pr_1 \colon C_{p+1}\times C_n\rightarrow C_{p+1}$ vanishes whenever $j \geq 1$. Thus
$$
H^m(C_{p+1},\wedge^nM_{p+1,B}\otimes \sE)=H^m(C_{p+1}\times C_n, (\sE\boxtimes N_{n,B})(-D_{p+1,n})).
$$
\end{remark}

Now, given a line bundle $L$ on the curve $C$, we consider the locus 
$$
\mathcal{L}_n^\ell(L) :=\{\xi\in C_n\mid h^0(C, L(-\xi))\geq \ell\}~~\text{ for integers $n, \ell \geq 0$}.
$$
When $\ell \geq 1$ and $n \geq r(L)-\ell+1$, the locus $\mathcal{L}_n^\ell(L)$ coincides with the secant space $V_n^{r(L)-\ell}(L)$ of $n$-secant $(r(L)-\ell)$-planes to $C$ in $\nP H^0(C, L)$ (see \cite[Definition 1.1]{CM}). In general, the locus $\mathcal{L}_n^{\ell}(L)$ is a closed subset of $C_n$ and there is a natural stratification
$$
C_n=\mathcal{L}^0_n(L)\supseteq \mathcal{L}^1_n(L)\supseteq \mathcal{L}^2_n(L)\supseteq\cdots.
$$ 
When $n=0$, the locus $\mathcal{L}^\ell_0(L)$  either contains a point and has dimension zero, or else is empty and has dimension $-1$ by convention. 

\begin{proposition}\label{prop:dimL^ell_n=>H^n=0}
Let $L$ be  a line bundle on $C$. Fix integers $m,n,p\geq 0$, and suppose that 	
$$
\dim \mathcal{L}^\ell_n(L)\leq p+n-m-\ell
$$
for all integers $\ell\geq 0$ with $p+1-m\leq \ell\leq p+1+n-m$.
Then we have the following:
	\begin{enumerate}
		\item $m\leq p$.
		\item If $\sE$ is a locally free sheaf on $C_n$, then
		$$H^m( C_{p+1}\times C_n, (N_{p+1,L}\boxtimes\sE)(-D_{p+1,n}))=0.$$		
		\item If $\sE$ is a locally free sheaf on $C_{n-i}$ for an integer $i \geq 0$, then
		$$H^{m-j}(C_{p+1}\times C_{n-i}, (N_{p+1,L}\boxtimes \sE )(-D_{p+1,n-i}))=0~~\text{ for $i\leq j\leq 2i$}.$$
	\end{enumerate}
\end{proposition}

\begin{proof} 
$(1)$ If $m \geq p+1$,  we can take $\ell=0$ and then $n=\dim \mathcal{L}^0_n(L)\leq p+n-m\leq n-1$, which is a contradiction. 

\medskip

\noindent $(2)$ Let $\pr_2 \colon C_{p+1}\times C_n\rightarrow C_n$ be the projection map. Recall that the dualizing sheaf of $C_{p+1}\times C_n$ is $\omega_{C_{p+1} \times C_n} = N_{p+1,\omega_C}\boxtimes N_{n,\omega_C}$. By Serre duality, the cohomology in $(2)$ is dual to 
\begin{equation}\label{eq:01}
H^{p+1+n-m}(C_{p+1}\times C_n, (S_{p+1,\omega_C\otimes L^{-1}}\boxtimes (\sE^\vee \otimes N_{n,\omega_C}))(D_{p+1,n})).
\end{equation}
The Leray spectral sequence for $\pr_2$ yields 
$$
E_2^{p+1+n-m-\ell,\ell}:=H^{p+1+n-m-\ell}(C_n, \sE^\vee \otimes N_{n,\omega_C}\otimes R^{\ell}\pr_{2,*}(S_{p+1,\omega\otimes L^{-1}}\boxtimes \sO_{C_n}(D_{n,p+1})))
$$
converging to the cohomology in (\ref{eq:01}). It is enough to show that
$$
E_2^{p+1+n-m-\ell, \ell} = 0~~\text{ for $0 \leq \ell \leq p+1+n-m$}.
$$
By $(1)$, we get $E_2^{p+1+n-m,0}=0$. For $\ell \geq 1$, it suffices to show that 
\begin{equation}\label{eq:dim R^ell}
\dim \Supp R^{\ell} \pr_{2,*}(S_{p+1,\omega\otimes L^{-1}}\boxtimes\sO_{C_n}(D_{p+1,n})) \leq p+n-m-\ell.
\end{equation}
To this end, note that the restriction of the sheaf $(S_{p+1,\omega\otimes L^{-1}}\boxtimes \sO_{C_n})(D_{p+1,n})$ to the fiber $\pr^{-1}_2(\xi)\cong C_{p+1}$ over $\xi\in C_n$ is $S_{p+1,\omega\otimes L^{-1}(\xi)}$. By Lemma \ref{lem:R^ipr_1}, we have
$$
H^{\ell}(C_{p+1},S_{p+1,\omega_C\otimes L^{-1}(\xi)})=S^{p+1-\ell}H^0(C,\omega_C\otimes L^{-1}(\xi))\otimes \wedge^{\ell}H^1(C,\omega_C\otimes L^{-1}(\xi))=0
$$
if $h^1(C,\omega_C\otimes L^{-1}(\xi))=h^0(C,L(-\xi))<\ell$. This suggests that
$$
\Supp R^{\ell} \pr_{2,*}(S_{p+1,\omega\otimes L^{-1}}\boxtimes\sO_{C_n}(D_{p+1,n}))\subseteq \mathcal{L}^{\ell}_n(L).
$$
By the assumption on the dimension of $\mathcal{L}^{\ell}_n(L)$, we obtain the inequality in (\ref{eq:dim R^ell}) and hence conclude that $E_2^{p+1+n-m-\ell,\ell}=0$.

\medskip

\noindent $(3)$ In view of $(2)$, it is sufficient to show that
\begin{equation}\label{eq:dimL^ell_n-i}
\dim \mathcal{L}^{\ell}_{n-i}(L)\leq p+n-i-m+j-\ell
\end{equation}
for all integers $\ell\geq 0$ with $p+1-m+j\leq \ell\leq p+1+n-i-m+j$. Recall that the addition map $\sigma:=\sigma_{n-i,1} \colon C_{n-i} \times C \to C_{n-i+1}$ sending $(\xi,x)$ to $\xi+x$ is a finite morphism. Observe that 
$$
\mathcal{L}^{\ell}_{n-i+1}(L)\subseteq \sigma(\mathcal{L}^{\ell}_{n-i}(L)\times C)\subseteq \mathcal{L}^{\ell-1}_{n-i+1}(L),
$$
which in particular implies that 
$$
\dim \mathcal{L}^{\ell}_{n-i}(L) \leq \dim \mathcal{L}^{\ell-1}_{n-i+1}(L)-1.
$$
Iterating this inequality, we obtain
$$
\dim \mathcal{L}^{\ell}_{n-i}(L) \leq \dim \mathcal{L}^{\ell-i}_{n}(L)-i.
$$
Recall that $\ell$ is an integer with $p+1-m+j\leq \ell\leq p+1+n-i-m+j$. Since $p\geq m$ and $i\leq j \leq 2i$, we get $p+1-m\leq\ell-i\leq p+1+n-m$ and $\ell \geq i+1$. We deduce (\ref{eq:dimL^ell_n-i}) from the given assumption as follows:
$$
\dim \mathcal{L}^{\ell}_{n-i}(L)\leq \dim \mathcal{L}^{\ell-i}_{n}(L)-i\leq p+n-m-\ell\leq p+n-i-m+j-\ell.
$$
This  completes the proof.
\end{proof}

\begin{corollary}\label{cor:01} Let $B$ be a $p$-very ample line bundle and $L$ a line bundle on $C$. Fix integers $m,n\geq 0$ and suppose that 
$$
\dim \mathcal{L}^\ell_n(L)\leq p+n-m-\ell
$$
for all integers $\ell\geq 0$ with $p+1-m\leq \ell\leq p+1+n-m$. Then 
$$
H^m(C_{p+1},\wedge^nM_{p+1,B}\otimes N_{p+1,L})=0.
$$
\end{corollary}

\begin{proof}
By Proposition \ref{prop:H^m(MotimesE)=0}, it suffices to check that
$$
H^{m-j}(C_{p+1} \times C_{n-i}, (N_{p+1, L} \boxtimes N_{n-i,B})(-D_{p+1, n-i}))=0
$$
for all $i,j$ with $0 \leq i \leq m-1$ and $i \leq j \leq 2i$. But these cohomology vanishing follow from Proposition \ref{prop:dimL^ell_n=>H^n=0}.
\end{proof}

We are ready to prove the following theorem which is the main technical result of this section. Notice that it does not require any positivity condition on $L$. 

\begin{theorem}\label{thm:01} 
Let $B$ be a $p$-very ample line bundle and $L$ a line bundle on $C$. Fix integers $m,n\geq 0$ with $m\leq p$. If
$h^0(C,L)\leq p+n-m$, then  
$$
H^m(C_{p+1},\wedge^nM_{p+1,B}\otimes N_{p+1,L})=0.
$$
\end{theorem}

\begin{proof}
It is enough to verify the hypothesis on the dimension of $\mathcal{L}^\ell_n(L)$ in Corollary \ref{cor:01}. To derive a contradiction, suppose that there exits an integer $\ell$ with $p+1-m\leq \ell\leq p+1+n-m$ such that 
	$$\dim \mathcal{L}^\ell_n(L)\geq p+n-m-\ell+1.$$
Set $d:=\deg L$, and consider the closed subset 
$$
C^{\ell}_{d-n}:=\{\xi'\in C_{d-n}\mid h^0(C, \sO_C(\xi'))\geq \ell\}\subseteq C_{d-n}.
$$
Write $\sigma_{n,d-n} \colon C_{n}\times C^{\ell}_{d-n}\rightarrow C_d$ for the addition map given by $(\xi',\xi) \mapsto \xi'+\xi$, and  consider the projection map $\pr_1 \colon C_n\times C^{\ell}_{d-n}\rightarrow C_n$. Observe that 
$$
\mathcal{L}^{\ell}_n(L)=\pr_{1}(\sigma_{n,d-n}^{-1}(|L|)).
$$
For any $\xi\in \mathcal{L}^{\ell}_n(L)$, we can identify $\pr_1^{-1}(\xi) = C^{\ell}_{d-n} \subseteq C_{d-n}$. 
Then 
$$
\pr_1^{-1}(\xi)\cap \sigma_{n,d-n}^{-1}(|L|)=|L(-\xi)|
$$
which has dimension $\geq \ell-1$. This means that the projection from $\sigma_{n,d-n}^{-1}(|L|)$ to $\mathcal{L}^{\ell}_n(L)$ has fibers of dimensions at least $\ell-1$. Since $\sigma_{n,d-n}$ is a finite morphism, we can count dimensions:
$$
\dim |L|\geq \sigma^{-1}_{n,d-n}(|L|)\geq \dim \mathcal{L}^{\ell}_n(L)+\ell-1\geq p+n-m.
$$
This gives a contradiction to the assumption that $h^0(C,L)\leq p+n-m$.
\end{proof}

\begin{corollary}\label{cor:h^1(C,L-B)<=r(B)-p-1}
Let $B$ be a $p$-very ample line bundle and $L$ a line bundle on $C$. Assume that $h^1(C, L \otimes B^{-1}) \leq r(B)-p-1$. Then we have
$$
H^1(C_{p+1}, M_{p+1,B}\otimes N_{p+1,L})=0.
$$
\end{corollary}

\begin{proof}
Recall that the kernel bundle $M_{p+1,B}$ is defined in the short exact sequence  (\ref{eq:sesforM_{p+1,B}}). 
%$$
%0 \longrightarrow M_{p+1,B} \longrightarrow H^0(C, B)\otimes \sO_{C_{p+1}} \longrightarrow E_{p+1,B} \longrightarrow 0.
%$$
Note that $\rank M_{p+1, B} = r(B)-p$ and $\det M_{p+1, B} = N_{p+1, B}^{-1}$. 
By Serre duality, it is enough to prove 
$$
H^p(C_{p+1}, \wedge^{r(B)-p-1} M_{p+1, B} \otimes N_{p+1, \omega_C \otimes B \otimes L^{-1}}) = 0.
$$
But the assumption $h^1(C, L \otimes B^{-1}) \leq r(B)-p-1$ implies
$$
h^0(C, \omega_C \otimes B \otimes L^{-1}) = h^1(C, L \otimes B^{-1}) \leq r(B)-p-1 = p+(r(B)-p-1)-p, 
$$
so the assertion follows from Theorem \ref{thm:01}.
\end{proof}

\begin{remark}\label{rem:h^1<->degL} 
The upper bound condition imposed on the specialty of $L\otimes B^{-1}$ in Corollary \ref{cor:h^1(C,L-B)<=r(B)-p-1} and Theorem \ref{main:03} can be transferred to one on the degree of  $L$, which turns out to be convenient in applications. Indeed,  by Riemann--Roch,
$$
h^1(C, L \otimes B^{-1}) = g-1-\deg L +\deg B + h^0(C, L \otimes B^{-1}),
$$
and
$$
r(B)-p-1 = \deg B - g+1 + h^1(C, B) -p-2.
$$
Therefore the inequality $h^1(C, L \otimes B^{-1}) \leq r(B)-p-1$ is equivalent to the inequality 
\begin{equation}\label{eq:degL}
\deg L \geq 2g + p + h^0(C, L \otimes B^{-1}) - h^1(C, B).
\end{equation}
\end{remark}

\begin{proof}[Proof of Theorem \ref{main:03}] 
We first show that $L$ is $p$-very ample. By Remark \ref{rem:h^1<->degL}, the assumption (\ref{eq:h^1(L-B)intro}) is equivalent to (\ref{eq:degL}). If $B$ is nonspecial, then it is clear that  $L$ is $p$-very ample. If $B$ is special, we prove by contradiction by assuming $L$ is not $p$-very ample. Thus there exists an effective divisor $\xi$ of degree $p+1$ such that $L(-\xi)$ is special. So we may write $B=\omega_C(-\alpha)$ and $L=\omega_C(\xi-\beta)$ for some effective divisors $\alpha$ and $\beta$ on $C$. We then have
$$
2g-1+p-\deg \beta = \deg L \geq 2g+p + h^0(C, \sO_C(\xi-\beta+\alpha)) - h^0(C, \sO_C(\alpha)) \geq 2g+p- \deg \beta,
$$
which is a contradiction. Thus in particular $L$ is globally generated. Now, the theorem is an immediate consequence of Lemma \ref{lem:H^1(MotimesN)=0=>K_{p,1}=0} and Corollary \ref{cor:h^1(C,L-B)<=r(B)-p-1}.
 \end{proof}

In the rest part of this section, we present several applications of Theorem \ref{main:03}. 

\begin{corollary}\label{cor:h^1(C,L-B)<=1}
Let $B$ be a $p$-very ample line bundle and $L$ a line bundle on $C$. Assume that $h^1(C, L \otimes B^{-1}) \leq 1$. Then 
$H^1(C_{p+1}, M_{p+1,B}\otimes N_{p+1,L}) \neq 0$ if and only if $h^1(C, L \otimes B^{-1}) =1$ and one of the following holds:
\begin{enumerate}
\item[(1)] $p=0$ and $|B|$ is a base point free pencil.
\item[(2)] $p=1$, $C \subseteq \nP H^0(C, B)= \nP^2$ is a plane curve of degree $\geq 3$, and $H^0(C, L \otimes B^{-1})\neq 0$. 
\item[(3)] $p\geq 2$, $g=1$, $\deg B = p+2$, and $L=B$.
\end{enumerate}
\end{corollary}

\begin{proof}
Suppose that $H^1(C_{p+1}, M_{p+1,B}\otimes N_{p+1,L}) \neq 0$. By Corollary \ref{cor:h^1(C,L-B)<=r(B)-p-1}, 
$$
r(B)-p \leq h^1(C, L \otimes B^{-1}) \leq 1,
$$
so $r(B) \leq p+1$. If $r(B) \leq p$, then $\rank M_{p+1, B}=0$ so that $H^1(C_{p+1}, M_{p+1,B}\otimes N_{p+1,L}) =0$. Thus $r(B) =p+1$, and $h^1(C, L \otimes B^{-1}) =1$. 

\medskip

Let us consider first the case when $C=\nP^1$. In this case, $B=\sO_{\nP^1}(p+1)$ and $L=\sO_{\nP^1}(p-1)$. Notice that if $p \geq 1$, then $L$ is globally generated and 
$$
H^1(C_{p+1}, M_{p+1,B}\otimes N_{p+1,L}) = K_{p,1}(\nP^1, B; L)=H^1(\nP^1, \wedge^{p+1} M_L \otimes B)=0.
$$
Hence the only possibility is that $p=0$ and $B=\sO_{\nP^1}(1)$ defines a base point free pencil. Next we consider arbitrary curves in the sequel by assuming $g\geq 1$. By  \cite[Proposition 3.4]{DNP}, we must have (1) $|B|$ is a base point free pencil, (2) $C \subseteq \nP H^0(C, B)= \nP^2$ is a plane curve of degree $\geq 3$, or (3) $g=1$, $\deg B = p+2$ for $p \geq 2$. In the last case, the condition $h^1(C, L \otimes B^{-1}) =1$ forces $L=B$ or $B(-x)$ for any $x \in C$. We will exclude the case of $L=B(-x)$. 
\medskip

Finally, it only remains to show that $(i)$ $H^1(C_{p+1}, M_{p+1,B}\otimes N_{p+1,L}) \neq 0$ really occurs under the circumstances (1), (2), (3) and $(ii)$ $H^1(C_{p+1}, M_{p+1,B}\otimes N_{p+1,L}) = 0$ when $g=1, \deg B=p+2, L=B(-x)$ for $x \in C$.
This comes from the observation that $\rank M_{p+1, B} = 1$ in all these cases. By Lemma \ref{lem:H^i(N)H^i(S)},
$$
H^1(C_{p+1}, M_{p+1,B}\otimes N_{p+1,L}) = H^1(C_{p+1}, S_{p+1, L \otimes B^{-1}}) = S^p H^0(C, L \otimes B^{-1}) \otimes H^1(C, L \otimes B^{-1}).
$$
Thus $H^1(C_{p+1}, M_{p+1,B}\otimes N_{p+1,L})  \neq 0$ if and only if either $p=0$ or $H^0(C, L \otimes B^{-1})\neq 0$.
\end{proof}

\begin{corollary}[{Green's $(2g+1+p)$-theorem}]
If $L$ is a line bundle on $C$ and $\deg L\geq 2g+p+1$, then $K_{p,2}(C;L)=0$. 
\end{corollary}

\begin{proof}
Since $\deg(L) \geq 2g+p+1$, it follows from Riemann--Roch theorem that $L$ is $p$-very ample and $r(L) \geq g+p+1$. Note that $K_{p,2}(C;L)=K_{p,1}(C,L;L)$. Since $h^1(C, L\otimes L^{-1})= g \leq r(L)-p-1$, the assertion follows from Theorem \ref{main:03}.
\end{proof}

The following is a well-known generalization of Castelnuovo's base point free pencil trick. 

\begin{corollary}[{Base point free pencil trick}]\label{cor:bpftrick}
Let $B$ be an effective divisor and $L$ a nonspecial globally generated line bundle. Then
$$
H^1(C, \wedge^{w+1} M_L\otimes B\otimes L)=0 ~~\text{ for $w\leq \deg L-2g$}.
$$
unless $B=\sO_C$, in which case $H^1(C, \wedge^w M_L \otimes L) = 0$ for $w \leq \deg L - 2g$.
\end{corollary}

\begin{proof} 
Note that $H^1(C, B \otimes L)=0$. 
We may assume that $w \geq 0$. As $\deg (L \otimes B) \geq 2g+w$, we see that $L \otimes B$ is $w$-very ample. If $B \neq \sO_C$, then 
$$
\deg L \geq 2g+w+\underbrace{h^0(C, L \otimes (L \otimes B)^{-1})}_{=0} - \underbrace{h^1(C, L \otimes B)}_{=0}.
$$
By Theorem \ref{main:03} and Remark \ref{rem:h^1<->degL}, we have $H^1(C, \wedge^{w+1} M_L\otimes B\otimes L) = K_{w,1}(C, B \otimes L; L) = 0$.
If $B = \sO_C$, then Green's $(2g+1+p)$-theorem yields $H^1(C, \wedge^w M_L \otimes L) = K_{w-1,2}(C; L) = 0$ when $\deg L \geq 2g+(w-1)+1$. 
\end{proof}

%======================New Section Starts================================
\section{Effective nonvanishing for weight-one syzygies of algebraic curves}\label{sec:nonvanishing}

\noindent  In this section, we discuss the nonvanishing of the Koszul cohomology $K_{p,1}(C, B; L)$. Here as in previous sections,  let $B$ and $L$ are line bundles on a smooth projective curve $C$ of genus $g \geq 0$. We are trying to relate positivity of line bundles with nonvanishing of the Koszul cohomology.

\medskip

We start by recalling the following very useful nonvanishing theorem for $K_{p,1}(C;L)$ proved by Green--Lazarsfeld. The original theorem holds for arbitrary smooth projective varieties, but we only state the version here for curves.

\begin{theorem}[{\cite[Appendix]{Green:KoszulI}}]\label{thm:GLnonvan}
If $L=M_1 \otimes M_2$ for some line bundles $M_1, M_2$ on $C$ such that $r_1:=r(M_1) \geq 1, r_2:=r(M_2) \geq 1$, then $K_{p, 1}(C; L) \neq 0$ for $1 \leq p \leq r_1+r_2-1$.
\end{theorem}

Together with Theorem \ref{main:03}, we obtain the following.

\begin{corollary}\label{cor:2g+p}
Assume $g \geq 2$.
Suppose that $\omega_C$ is $p$-very ample and $\deg L \geq 2g+p$. Then
$$
K_{p,1}(C, \omega_C; L) \neq 0~ \text{ if and only if } ~h^1(C, L \otimes \omega_C^{-1}) \geq g-p-1.
$$
Furthermore, we have the following:
\begin{enumerate}
	\item If $\omega_C$ is not $(p+1)$-very ample, then $K_{p+1, 1}(C, \omega_C; L) \neq 0$ except for the case when $g=2$, $p=0$, and $\deg L = 4$, in which case, $K_{1,1}(C,\omega_C;L)=0$. 
	\item If $\deg L =2g+p$, the condition $h^1(C, L \otimes \omega_C^{-1}) \geq g-p-1$ is equivalent to $L \otimes \omega_C^{-1}$ computes the gonality $\gon(C)=p+2$.
\end{enumerate}

\end{corollary}

\begin{proof}  
If $h^1(C, L \otimes \omega_C^{-1}) \leq g-p-2$, then Theorem \ref{main:03} shows that $K_{p,1}(C, \omega_C; L)=0$. Conversely, suppose that $h^1(C, L \otimes \omega_C^{-1}) \geq g-p-1$. Set $d:=\deg L$. As $h^0(C, L) = d-g+1$, the duality of Koszul cohomology (\cite[Theorem 2.c.1]{Green:KoszulI}) gives
$$
K_{p,1}(C, \omega_C; L) = K_{d-g-p-1, 1}(C; L)^{\vee}.
$$ 
Consider two line bundles $M_1:=\omega_C$ and $M_2:=L \otimes \omega_C^{-1}$. We compute that  $r(M_1)= g-1$ and 
$$r(M_2)=\deg (L \otimes \omega_C^{-1}) -g+h^1(C, L \otimes \omega_C^{-1}).$$
The hypothesis imposed on $\deg L$ and $h^1(C,L\otimes \omega_C^{-1})$ gives $r(M_2)\geq d-p-2g+1\geq 1$. By Green--Lazarsfeld's nonvanishing theorem (Theorem \ref{thm:GLnonvan}), we find $K_{d-p-g-1, 1}(C; L) \neq 0$ and therefore $K_{p,1}(C, \omega_C; L)\neq 0$ as desired.

\medskip

Now, we prove the statement (1) by assuming that $\omega_C$ is not $(p+1)$-very ample. There exists an effective divisor  $\xi$ of degree $p+2$ such that $h^0(C, \sO_C(\xi)) = h^1(C, \omega_C(-\xi))=2$.  As before, the duality of Koszul cohomology (\cite[Theorem 2.c.1]{Green:KoszulI}) gives
$$
K_{p+1,1}(C, \omega_C; L) = K_{d-g-p-2, 1}(C; L)^{\vee}.
$$ 
For $M_1:=\sO_C(\xi)$ and $M_2:=L(-\xi)$, we have $r(M_1)=1$ and $r(M_2) \geq d-p-2-g$. If $d-p-2-g\geq 1$, then Theorem \ref{thm:GLnonvan} shows $K_{d-p-g-2, 1}(C; L) \neq 0$. On the other hand, note that $d-p-2-g \geq g-2$. Thus if $d-p-2-g = 0$, then $g=2$ and $p=0$. In this case, $d=\deg L = 4$.  The duality above then gives $K_{1,1}(C,\omega_C;L)=K_{0,1}(C;L)^{\vee}=0$.

\medskip

Finally, we prove the statement (2) by assuming that $\deg L = 2g+p$. Then $\deg (L \otimes \omega_C^{-1}) = p+2$. By
Riemann--Roch,
$$
h^0(C, L \otimes \omega_C^{-1}) = p+g+3+ h^1(C, L \otimes \omega_C^{-1}).
$$
Thus $h^1(C, L \otimes \omega_C^{-1}) \geq g-p-1$ if and only if $h^0(C, L \otimes \omega_C^{-1}) \geq 2$.
However, as $\gon(C) \geq p+2$, we have $h^0(C, L \otimes \omega_C^{-1}) \leq 2$. Hence $h^0(C, L \otimes \omega_C^{-1}) \geq 2$ is equivalent to that the gonality $\gon(C)=p+2$ and $L\otimes \omega^{-1}_C$ computes the gonality. %$h^0(C, L \otimes \omega_C^{-1}) = 2$  and $\gon (C)=p+2$.
\end{proof}

Using the corollary, we may prove the effective gonality theorem (Theorem \ref{thm:effgon}) right away, but we will derive it from a more general statement (Theorem \ref{main:04}) in the next section. In order to prove Theorem \ref{main:04}, in the rest part of this section, we will prepare  effective nonvanishing results for the Koszul cohomology $K_{p,1}(C, B; L)$ when $B$ is neither $\sO_C$ nor $\omega_C$, which are the situations where Green--Lazarsfeld's theorem says nothing.

\begin{lemma}\label{lem:K_{l-1,1}=>K_{l,1}} 
Let $B$ be line bundle and $L$ a globally generated line bundle on $C$. Assume that there is a point $x \in C$ such that $B$ separates $x$ and $L(-x)$ is globally generated such that
	$$
	H^1(C, L(-x))=0~~\text{ and }~~H^1(C, \wedge^{\ell} M_{L(-x)} \otimes B \otimes L(-x))=0.
	$$ 
Then there is a surjective map 
$$
K_{\ell,1}(C,B;L) \longrightarrow K_{\ell-1,1}(C,B(-x);L(-x)).
$$
In particular, 
$$
K_{\ell-1,1}(C, B(-x); L(-x)) \neq 0 \Longrightarrow K_{\ell, 1}(C, B; L) \neq 0.
$$
\end{lemma}

\begin{proof}
Let $\pr_1 \colon C_{\ell+1} \times C \to C_{\ell+1}$ and $\pr_2 \colon C_{\ell+1} \times C \to C$ be projection maps. Note that $\mathcal{D}_x:=\pr_1^{-1}(S_{\ell+1, x}) = C_{\ell} \times C$. We have a commutative diagram \\[-25pt]
	
	\begin{footnotesize}
		$$
		\xymatrixcolsep{0.3in}
		\xymatrix{
			0 \ar[r] &  (N_{\ell+1, L(-x)} \boxtimes B)(-D_{\ell+1}) \ar[r]^-{\mathcal{D}_x} \ar[d] & (N_{\ell+1, L} \boxtimes B)(-D_{\ell+1})  \ar[r] \ar[d] & (N_{\ell, L(-x)} \boxtimes B(-x))(-D_{\ell}) \ar[r] \ar[d] & 0\\
			0 \ar[r] & N_{\ell+1, L(-x)} \boxtimes B \ar[r]^-{\mathcal{D}_x} & N_{\ell+1, L} \boxtimes B \ar[r] & N_{\ell, L(-x)} \boxtimes B \ar[r] & 0,
		}
		$$
	\end{footnotesize}
	
	\noindent where the rows are short exact sequences. Since $H^1(C, L(-x))=0$, it follows from Lemma \ref{lem:H^i(N)H^i(S)} that
	$$
	R^1 \pr_{2,*}(N_{\ell+1, L(-x)} \boxtimes B)(-D_{\ell+1})=0~~\text{ and }~~R^1\pr_{2,*}N_{\ell+1, L(-x)} \boxtimes B =0.
	$$
	Applying $\pr_{2,*}$ yields a commutative diagram\\[-20pt]
	
	\begin{footnotesize}
		$$
		\xymatrix{
			0\ar[r]&\wedge^{\ell+1} M_{L(-x)} \otimes B \ar[r] \ar[d] & \wedge^{\ell+1} M_L \otimes B \ar[r] \ar[d] & \wedge^{\ell} M_{L(-x)} \otimes B(-x)\ar[d] \ar[r]&0\\
			0\ar[r]& \wedge^{\ell+1} H^0(C, L(-x)) \otimes B\ar[r] & \wedge^{\ell+1} H^0(C, L)\otimes B \ar[r] & \wedge^{\ell} H^0(C, L(-x))\otimes B \ar[r] &0,
		}
		$$
	\end{footnotesize}
	
	\noindent where the rows are short exact sequences. Taking cohomology, we get a commutative diagram\\[-25pt]
	
	\begin{footnotesize}
		$$
		\xymatrixcolsep{0.25in}
		\xymatrix{
			H^1(C, \wedge^{\ell+1} M_{L(-x)} \otimes B) \ar[r] \ar@{->>}[d]^-{\alpha} & H^1(C, \wedge^{\ell+1} M_L \otimes B) \ar@{->>}[r] \ar[d]^-{\beta} & H^1(C, \wedge^{\ell} M_{L(-x)} \otimes B(-x))\ar[d]^{\gamma} \\
			H^1(C, B) \otimes \wedge^{\ell+1} H^0(C, L(-x)) \ar[r] & H^1(C, B) \otimes \wedge^{\ell+1} H^0(C, L) \ar@{->>}[r] & H^1(C, B) \otimes \wedge^{\ell} H^0(C, L(-x)),
		}
		$$
	\end{footnotesize}
	
	\noindent where the rows are exact sequences. Since $\coker(\alpha)=H^1(C, \wedge^{\ell} M_L(-x) \otimes B \otimes L(-x))=0$ by the given hypothesis, the map $\alpha$ is surjective. The map $\gamma$ factors as\\[-25pt]
	
	\begin{small}
		$$H^1(C, \wedge^{\ell} M_{L(-x)} \otimes B(-x))\longrightarrow H^1(C, B(-x)) \otimes \wedge^{\ell} H^0(C, L(-x))\xrightarrow{~j~} H^1(C, B) \otimes \wedge^{\ell} H^0(C, L(-x)).$$
	\end{small} 
	
	\noindent Since $B$ separates $x$, the map $j$ is injective. Thus the kernel of $\gamma$ is $K_{\ell-1,1}(C,B(-x);L(-x))$. A diagram chasing gives a surjection map $\ker(\beta) \to \ker(\gamma)$, where $\ker(\beta)=K_{\ell, 1}(C, B; L)$.
	\end{proof}

\begin{corollary}\label{cor:K_{0,1}=>K_{l,1}}
Let $B$ be line bundle and $L$ a globally generated line bundle on $C$. Assume that there exist not necessarily distinct points $x_1, \ldots, x_{\ell} \in C$ such that if $\xi_i:=x_1 + \cdots + x_i$ for $1 \leq i \leq \ell$ and $\xi_0 := 0$, then $B$ separates $\xi_{\ell}$ and $L(-\xi_i)$ is globally generated with
	$$
H^1(C, L(-\xi_\ell))=0~~\text{ and }~~H^1(C, \wedge^{\ell+1-i} M_{L(-\xi_i)} \otimes B(-\xi_{i-1}) \otimes L(-\xi_i)) = 0 \text{ for every  }1 \leq i \leq \ell.
	$$
	Then
	$$
	K_{0,1}(C, B(-\xi_{\ell}); L(-\xi_{\ell})) \neq 0 \Longrightarrow K_{\ell, 1}(C, B; L) \neq 0.
	$$
\end{corollary}

\begin{proof}
	By induction on $\ell$, the corollary follows from Lemma \ref{lem:K_{l-1,1}=>K_{l,1}}
\end{proof}

\begin{lemma}\label{lem:K_{0,1}I}
Let $B$ and $L$ be line bundles on $C$. We have
	$$K_{0,1}(C,B;L)\neq 0$$
	if one the following holds:
	\begin{enumerate}
		\item $H^0(C, B)=0$ and $H^0(C, B\otimes L)\neq 0$.
		\item $|B|$ has a nonempty base locus and $H^1(C, L)=0$.
	\end{enumerate}
\end{lemma}

\begin{proof}
The lemma easily follows from the fact that $K_{0,1}(C, B; L) \neq 0$ if and only if the multiplication map
$H^0(C, B) \otimes H^0(C, L) \rightarrow H^0(C, B \otimes L)$ is not surjective.
\end{proof}

\begin{proposition}\label{prop:K_{p+1,1}}
Let $B$ be a line bundle and $L$ a globally generated line bundle on $C$. Suppose that $B$ is $p$-very ample but not $(p+1)$-very ample so that there exist not necessarily distinct points $x_1, \ldots, x_{p+1} \in C$ such that $B(-\xi_{p+1})$ is not globally generated, where $\xi_i:=x_1 + \cdots + x_i$ for $1 \leq i \leq p+1$ and $\xi_0 := 0$. Assume that $L(-\xi_{p+1})$ is globally generated with $H^1(C, L(-\xi_{p+1}))=0$ and 
	$$
	H^1(C, \wedge^{p+2-i} M_{L(-\xi_i)} \otimes B(-\xi_{i-1}) \otimes L(-\xi_i)) = 0~~\text{ for every $1 \leq i \leq p+1$}.
	$$
Then $K_{p+1,1}(C, B;L) \neq 0$ except that $K_{p+1,1}(\nP^1,\sO_{\nP^1}(p);\sO_{\nP^1}(p+1))=0$ with $p \geq 1$.	
\end{proposition}

\begin{proof} 
Put $\xi:=\xi_{p+1}$. 
Since $B$ is $p$-very ample, we have $h^0(C, B)\geq p+1$. We first consider the case when $h^0(C, B)=p+1$. By \cite[Proposition 3.4]{DNP}, we have the following two cases:

\smallskip
	
\noindent \emph{Case} 1: $p=0$ and $B=\sO_C$. By assumption, $L(-x_1)$ is globally generated. If $h^0(C, L(-x_1))=1$, then $L=\sO_C(x_1)$. But $L$ is also globally generated, so we must have $C=\nP^1$ and $L=\sO_{\nP^1}(1)$, and thus $K_{1,1}(\nP^1,\sO_{\nP^1};\sO_{\nP^1}(1))=0$. If $h^0(C, L(-x_1))\geq 2$, then $h^0(C, L(-2x_1))>0$. By Lemma \ref{lem:K_{0,1}I}, $K_{0,1}(C,\sO_C(-x_1);L(-x_1))\neq 0$, and thus, by Lemma \ref{lem:K_{l-1,1}=>K_{l,1}}, we conclude $K_{1,1}(C,\sO_C;L)\neq 0$.
	
\smallskip 

\noindent \emph{Case} 2: $p\geq 1$, $C=\nP^1$ and $B=\sO_{\nP^1}(p)$. In this case, $L=\sO_{\nP^1}(m)$ with $m\geq p+1$. If $m=p+1$, then we have $K_{p+1,1}(\nP^1,\sO_{\nP^1}(p);\sO_{\nP^1}(p+1))=0$. If $m\geq p+2$, then $K_{0,1}(\nP^1,B(-\xi);L(-\xi))=K_{0,1}(\nP^1,\sO_{\nP^1}(-1);\sO_{\nP^1}(m-p-1))\neq 0$. Corollary \ref{cor:K_{0,1}=>K_{l,1}} gives $K_{p+1,1}(\nP^1,\sO_{\nP^1}(p);\sO_{\nP^1}(m))\neq 0$.

\smallskip

\noindent Now, we consider the case when $h^0(C, B)\geq p+2$. The line bundle $B(-\xi)$ is effective and the linear system $|B(-\xi)|$ has nonempty base locus. By Lemma \ref{lem:K_{0,1}I}, $K_{0,1}(C,B(-\xi);L(-\xi))\neq0$, and thus, by Corollary \ref{cor:K_{0,1}=>K_{l,1}}, we conclude $K_{p+1,1}(C,B;L)\neq 0$. 
\end{proof}

Now, we obtain a slight improvement of \cite[Proposition 3.5]{Agostini:AsymSyz}.

\begin{corollary}\label{cor:K_{p+1,1}}
Let $B$ and $L$ be line bundles on $C$. Suppose that $B$ is $p$-very ample but not $(p+1)$-very ample. 
If $\deg L\geq 2g+p+1$ then $K_{p+1,1}(C, B;L) \neq 0$ except for the following:
\begin{enumerate}
\item[(1)] $g=0$, $B=\sO_{\nP^1}(p)$, $L=\sO_{\nP^1}(p+1)$ with $p \geq 1$. In this case, $K_{p+1,1}(\nP^1,B; L)=0$.
\item[(2)] $g=1$, $B=\sO_C$, $\deg L=3$. In this case, $K_{1,1}(C; L)=0$.
\end{enumerate}
\end{corollary}
	
\begin{proof}
There exist  points $x_1, \ldots, x_{p+1} \in C$ such that $B(-\xi_{p+1})$ is not globally generated, where $\xi_i:=x_1 + \cdots + x_i$ for $1 \leq i \leq p+1$ and $\xi_0 := 0$. Note that $L(-\xi_{p+1})$ is globally generated and $H^1(C, L(-\xi_{p+1}))=0$. Thanks to Proposition \ref{prop:K_{p+1,1}}, we only need to check that
\begin{equation}\label{eq:H^1(wedge^{p+2-i}M_L(-xi_i))}
H^1(C, \wedge^{p+2-i} M_{L(-\xi_i)} \otimes B(-\xi_{i-1}) \otimes L(-\xi_i)) =0~~\text{ for every $1 \leq i \leq p+1$}.
\end{equation}
Note that $B(-\xi_{i-1})$ is effective and $B(-\xi_{i-1}) \otimes L(-\xi_i)$ is $(p+1-i)$-very ample. We have
$$
\deg L(-\xi_i) \geq 2g+(p+1-i).
$$
If $B(-\xi_p) \neq \sO_C$, then Corollary \ref{cor:bpftrick} yields (\ref{eq:H^1(wedge^{p+2-i}M_L(-xi_i))}) as desired. On the other hand, if $B(-\xi_p)=\sO_C$, then $B=\sO_C(\xi_p)$ and $h^0(C, B)=p+1$. By \cite[Proposition 3.4]{DNP}, we have either $p\geq 1$, $C=\nP^1$, $B=\sO_{\nP^1}(p)$ or $p=0$, $B=\sO_C$. For the former case, Case 2 in the proof of Proposition \ref{prop:K_{p+1,1}} yields that $K_{p+1,1}(\nP^1, B;L) = 0$ if and only if $L=\sO_{\nP^1}(p+1)$. For the latter case, consider the Koszul-type complex
$$
\wedge^2 H^0(C, L) \stackrel{\delta_1}\longrightarrow H^0(C, L)\otimes H^0(C, L)\stackrel{\delta_2}{\longrightarrow} H^0(C, L^2)
$$
computing $K_{1,1}(C;L)$. Here $\delta_1$ is clearly injective, and $\delta_2$ is surjective by Green's $(2g+1+p)$-theorem.
Putting $d:=\deg L \geq 2g+1$, we can directly compute 
$$
\dim K_{1,1}(C; L) = \frac{(d-g)^2-d-g}{2} \geq \frac{g^2-g}{2}.
$$ 
Thus $K_{1,1}(C; L) = 0$ if and only if $g=1$ and $d=3$. 
\end{proof}

Now, we want to apply Corollary \ref{cor:K_{0,1}=>K_{l,1}} to study when $K_{p,1}(C, B; L) \neq 0$ under the assumption that $B$ is $p$-very ample. 
Note that $B(-\xi)$ is globally generated for every $\xi:=x_1 + \cdots + x_p \in C_p$. Suppose that $L$ satisfies the conditions in Corollary \ref{cor:K_{0,1}=>K_{l,1}}. We have
$$
K_{0,1}(C, B(-\xi); L(-\xi)) = K_{0,1}(C, L(-\xi); B(-\xi)) = H^1(C, M_{B(-\xi)} \otimes L(-\xi)).
$$
Theorem \ref{main:03} says that if $h^1(C, L \otimes B^{-1}) \leq r(B)-p-1$, then $K_{p,1}(C, B; L) = 0$. Thus the problem is to find a condition for $K_{0,1}(C, L(-\xi);B(-\xi)) \neq 0$  assuming that
$$
h^1(C, L(-\xi) \otimes B(-\xi)^{-1}) = h^1(C, L \otimes B^{-1}) \geq r(B)-p = r(B(-\xi)).
$$

\begin{lemma}\label{lem:K_{0,1}II}
Let $B$ and $L$ be globally generated line bundles on $C$. Assume that $h^1(C, L)=0$ and $h^1(C, L \otimes B^{-1}) \geq r(B)$. 
	\begin{enumerate}
		\item If $B=\sO_C$, then $K_{0,1}(C;L)=0$.
		\item If $B\neq \sO_C$ and $H^0(C, L \otimes \omega_C^{-1}) \neq 0$, then $K_{0,1}(C,B;L)=K_{0,1}(C, L; B) \neq 0$.
	\end{enumerate}
\end{lemma}

\begin{proof}
As the assertion (1) is trivial, we only need to prove the assertion (2).
	We have
	$$
	K_{0,1}(C, L; B) = H^1(C, M_B \otimes L) = H^0(C, \wedge^{r(B)-1} M_B \otimes \omega_C \otimes B \otimes L^{-1})^{\vee}.
	$$
	As $H^0(C, \omega_C \otimes L^{-1})= 0$, we have
	$$
	H^0(C, \wedge^{r(B)-1} M_B \otimes \omega_C \otimes B \otimes L^{-1}) = K_{r(B)-1, 0}(C, \omega_C \otimes B \otimes L^{-1}; B).
	$$
Since $h^0(C, \omega_C \otimes B \otimes L^{-1}) = h^1(C, L \otimes B^{-1}) \geq r(B)$, we can choose linearly independent sections $s_1, \ldots, s_{r(B)} \in H^0(C, \omega_C \otimes B \otimes L^{-1})$. We can also take  a nonzero section $t \in H^0(C, L \otimes \omega_C^{-1})$. Then
	$$
	\sum_{i=1}^{r(B)} (-1)^i (s_1t \wedge \cdots \wedge \widehat{s_i t} \wedge \cdots \wedge s_{r(B)}t) \otimes s_i \in \wedge^{r(B)-1} H^0(C, B) \otimes H^0(C, \omega_C \otimes B \otimes L^{-1})
	$$
	is killed by the Koszul differential, so this yields $K_{r(B)-1, 0}(C, \omega_C \otimes B \otimes L^{-1}; B) \neq 0$.
\end{proof}

\begin{corollary}\label{cor:K_{p,1}}
Let $B$ be a $p$-very ample line bundle with $r(B) \geq p+1$ and $L$ a line bundle on $C$ with $h^1(C, L \otimes B^{-1}) \geq r(B)-p$. Assume that there exist not necessarily distinct points $x_1, \ldots, x_p \in C$ such that if $\xi_i:=x_1 + \cdots + x_i$ for $1 \leq i \leq p$ and $\xi_0:=0$, then $L(-\xi_p)$ is globally generated with $H^1(C, L(-\xi_p))=0$, $H^0(C, L(-\xi_p) \otimes \omega_C^{-1}) \neq 0$, and
\begin{equation}\label{eq:vanconforK_{p,1}}
	H^1(C, \wedge^{p+1-i} M_{L(-\xi_i)} \otimes B(-\xi_{i-1}) \otimes L(-\xi_i)) = 0~~\text{ 	for every $1 \leq i \leq p$.}
\end{equation}
Then $K_{p,1}(C, B; L) \neq 0$. In particular, if $\deg L \geq 2g+p$ with $H^0(C, L(-\xi) \otimes \omega_C^{-1}) \neq 0$ for some $\xi \in C_p$, then $K_{p,1}(C, B; L) \neq 0$.
\end{corollary}

\begin{proof}
Note that $h^0(C, B) \geq p+2$ implies $B(-\xi_p) \neq \sO_C$. 
The first part of the corollary then follows from Corollary \ref{cor:K_{0,1}=>K_{l,1}} and Lemma \ref{lem:K_{0,1}II}. For the second part, write $\xi=x_1 + \cdots + x_p$. It is enough to check (\ref{eq:vanconforK_{p,1}}) when $\xi_i:=x_1 + \cdots + x_i$ for $1 \leq i \leq p$ and $\xi_0:=0$. But this follows from Corollary \ref{cor:bpftrick} since $\deg L(-\xi_i) \geq 2g+(p-i)$.
\end{proof}

As an interesting quick application, we recover one part of Green--Lazarsfeld's result \cite[Theorem 1.2]{Lazarsfeld:SomeResults} on the failure of $N_p$-property for a line bundle of large degree. 

\begin{corollary}[{Green--Lazarsfeld}]
Assume that $g \geq 1$. Let $L$ be a line bundle on $C$ of degree $2g+p+1$. If $H^0(C, L \otimes \omega_C^{-1}) \neq 0$, then $K_{p+1, 2}(C; L) \neq 0$. 
\end{corollary}

\begin{proof}
Notice that $L$ is $(p+1)$-very ample, $r(L) = g+p+1 \geq p+2$, and $h^1(C, L \otimes L^{-1})=r(L)-(p+1)$. As $L\otimes \omega^{-1}_C$ is effective and has degree $p+3$, there is an effective divisor $\xi \in C_{p+1}$ such that $H^0(C, (L \otimes \omega_C^{-1})(-\xi)) \neq 0$. Then Corollary \ref{cor:K_{p,1}} yields $K_{p+1, 2}(C; L) \neq 0$. 
\end{proof}

%======================New Section Starts================================
\section{Effective gonality theorem}\label{sec:effgonthm}

\noindent In this section, we prove Theorem \ref{main:04} and deduce Theorem \ref{thm:effgon} and Corollary \ref{main:02} from it. We assume that $C$ is a smooth projective curve of genus $g\geq 2$. 

\begin{proof}[Proof of Theorem \ref{main:04}] 

As the assertion is trivial when $B = \sO_C$, we assume that $B \neq \sO_C$. The line bundle $B$ is $p$-very ample, so  we have $r(B)\geq p+1$ by the classification in \cite[Proposition 3.4]{DNP}. Observe that if $\xi$ is an effective divisor of degree $p$, then $|B(-\xi)|$ is a nonzero base point free linear system  and thus $\gon(C)\leq \deg B-p$. This particularly implies that $\Cliff(C)\leq \deg B-p-2$ (recall that Coppens--Martens \cite[Theorem 2.2]{CM} proved that $\Cliff(C)=\gon(C)-2$ or $\gon(C)-3$). Using this observation and  Riemann--Roch theorem applied for the line bundle $B$, we deduce from the second inequality in the hypothesis the following inequality
\begin{equation}\label{eq:Lbound}
	\deg L\geq 2g+3p+4-2h^0(C,B).
\end{equation}
By Theorem \ref{main:03} and Remark \ref{rem:h^1<->degL}, if
$$
\deg L \geq 2g + p + h^0(C, L \otimes B^{-1}) - h^1(C, B),
$$
then we  immediately  have $K_{i,1}(C, B; L) = 0$ for $0 \leq i \leq p$. Thus in the sequel, we focus on the case that 
\begin{equation}\label{eq:thm04-1}
	\deg L \leq 2g + p + h^0(C, L \otimes B^{-1}) - h^1(C, B)-1,
\end{equation}
which is equivalent to $h^1(C, L \otimes B^{-1}) \geq r(B)-p$ by Remark \ref{rem:h^1<->degL}. Then  $h^1(C, L \otimes B^{-1}) \geq 1$.
Furthermore, as $\deg L\geq 2g+p+1-h^1(C,B)$, the inequality (\ref{eq:thm04-1}) also implies that $h^0(C,L\otimes B^{-1})\geq 2$.

\medskip

If $h^1(C, L \otimes B^{-1})=1$, then Corollary \ref{cor:h^1(C,L-B)<=1} gives the exceptional cases (1) and (2) in the theorem. Notice also that in both cases the line bundle $L$ is nonspecial because of the degree bound of $L$ in (\ref{eq:Lbound}).

\medskip

Now, we consider the case when $h^1(C, L \otimes B^{-1}) \geq 2$. So the line bundle $L \otimes B^{-1}$ contributes to the Clifford index $\Cliff(C)$ of $C$ (the existence of a line bundle contributing $\Cliff(C)$ implies $g \geq 4$). This means that   
\begin{equation}\label{eq:thm04-3}
	(\deg L - \deg B) - 2h^0(C, L \otimes B^{-1}) + 2 = \Cliff(L \otimes B^{-1}) \geq \Cliff(C).
\end{equation}
Rewriting (\ref{eq:thm04-1}) as
\begin{equation}\label{eq:thm04-2}
	h^0(C, L \otimes B^{-1}) \geq \deg L -2g-p+1 + h^1(C, B)
\end{equation}
and combinng this with (\ref{eq:thm04-3}), we find
\begin{equation}\label{eq:thm04-4}
	(\deg L - \deg B) - \Cliff(C) + 2 \geq 2h^0(C, L \otimes B^{-1}) \geq 2\deg L -4g-2p+2 + 2h^1(C, B).
\end{equation}
Hence we obtain
\begin{equation}\label{eq:thm04-5}
	4g+2p-2h^1(C, B) - \Cliff(C) \geq \deg L + \deg B.
\end{equation}
However, the assumption on $\deg B+\deg L$ in the hypothesis of the theorem forces  (\ref{eq:thm04-5}) to be an equality. Consequently, all the inequalities in (\ref{eq:thm04-4}), (\ref{eq:thm04-2}), (\ref{eq:thm04-3}), (\ref{eq:thm04-1}) are equalities, and $h^1(C, L \otimes B^{-1}) = r(B)-p$. In particular, $\deg L = 2g+p+h^0(C, L \otimes B^{-1})-h^1(C, B)-1$, $\deg B =2g+p-h^0(C, L \otimes B^{-1}) - h^1(C, B)+1-\Cliff(C)$, and $L \otimes B^{-1}$ computes $\Cliff(C)$. This is the case (3) stated in the theorem. Furthermore, as $h^1(C, L \otimes B^{-1}) = r(B)-p$, we obtain $K_{i,1}(C, B; L) = 0$ for $0 \leq i \leq p-1$ by Theorem \ref{main:03}.

\medskip

Turning to the exceptional case $(3)$ with $B=\omega_C$, we have
$$
2g-2 = \deg \omega_C = 2g+p-h^0(C, L \otimes B^{-1})-\Cliff(C) \leq 2g-2+(p-\Cliff(C)).
$$
Thus $\Cliff(C) \leq p$. As $p+2 \leq \gon(C)$, we get $p-1 \leq \Cliff(C) \leq p$ (recall that $\Cliff(C)=\gon(C)-2$ or $\gon(C)-3$). Consider the case when $\Cliff(C)=p$. Then $h^0(C, L \otimes \omega_C^{-1})=2$, and $\deg L = 2g+p$. Hence $\deg (L \otimes \omega_C^{-1})=p+2$, so $\gon(C)=p+2$. Since $h^1(C, L \otimes \omega_C^{-1}) = g-p-1$, Corollary \ref{cor:2g+p} shows that $L = \omega_C(\xi)$ for some $\xi \in C_{p+2}$ with $\dim |\xi| = 1$. Consider the case when $\Cliff(C)=p-1$ (hence $\gon(C)=p+2$). Then $h^0(C, L \otimes B^{-1})=3$, and $\deg L = 2g+p+1$. We have $\gamma(C) = 2$. Here
$$
\gamma(C):=\min\{r(A) \mid \text{$A$ computes $\Cliff(C)$}\}
$$
is the \emph{Clifford dimension} of $C$. It is well-known that $\gamma(C)=1$  if and only if $\Cliff(C)=\gon(C)-2$. Since $H:=L \otimes \omega_C^{-1}$ computes $\gamma(C)$, it follows that $H$ gives an embedding $C \subseteq \nP H^0(C, H) = \nP^2$. Then $\omega_C = H^p$, and $L=\omega_C \otimes H$. 
Now, in any event, $\deg L \geq 2g+p$ and $h^1(C, L \otimes \omega_C^{-1}) = g-p-1$. Furthermore, $r(L) \geq g+p \geq p+1$ and $H^0(C, L(-\xi) \otimes \omega_C^{-1}) \neq 0$ for some $\xi \in C_p$. Thus Corollary \ref{cor:K_{p,1}} yields $K_{p,1}(C, \omega_C; L) \neq 0$ (instead one can also apply Corollary \ref{cor:2g+p} to conclude $K_{p,1}(C, \omega_C; L) \neq 0$).

\medskip

Finally, if $B$ is not $(p+1)$-very ample and $\deg L \geq 2g+p+1$, Corollary \ref{cor:K_{p+1,1}} shows that $K_{p+1, 1}(C, B; L) \neq 0$ since we assume $g \geq 2$.
\end{proof}

\begin{remark}
Theorem \ref{main:04} holds true even when $g=0$ or $1$. However, more exceptional cases occur such as Corollaries \ref{cor:h^1(C,L-B)<=1} and \ref{cor:K_{p+1,1}}. We leave the details to the interested readers.
\end{remark}

The following is essentially equivalent to the effective gonality theorem (Theorem \ref{thm:effgon}).

\begin{theorem}\label{main:01} 
Let $C$ be a smooth projective curve of genus $g \geq 2$, and  $L$ be a line bundle on $C$.
Assume that $\omega_C$ is $p$-very ample.
	\begin{enumerate}
		\item [(1)]  If $\deg L\geq 2g+p+2$, then $K_{p,1}(C, \omega_C; L) = 0$. 
		\item [(2)] If $\deg L = 2g+p+1$, then $K_{p,1}(C, \omega_C; L) \neq 0$ if and only if $C \subseteq  \nP H^0(C, H) = \nP^2$ is a plane curve of degree $p+3 \geq 4$ and $L=\omega_C \otimes H$. In this case, $p=\gon(C)-2$. 
		\item [(3)] If $\deg L = 2g+p$, then $K_{p,1}(C, \omega_C; L) \neq 0$ if and only if $L = \omega_C(\xi)$ for some $\xi \in C_{p+2}$ with $\dim |\xi| = 1$. In this case, $p=\gon(C)-2$. 
	\end{enumerate}
\end{theorem}

\begin{proof}
As $\gon(C) \geq p+2$, we have $\Cliff(C) \geq p-1$.
When $\deg L \geq 2g+p+1$, we have
$$
\deg \omega_C + \deg L \geq 4g+p-1 \geq 4g+2p-2-\Cliff(C).
$$
By Theorem \ref{main:04}, $K_{p,1}(C, \omega_C; L) \neq 0$ if and only if $\deg L = 2g+p+1$, $\Cliff(C)=p-1$, and $C \subseteq \nP H^0(C, H)=\nP^2$ is a plane curve of degree $p+3 \geq 4$ with $L=\omega_C \otimes H$. We have shown $(1)$ and $(2)$. When $\deg L = 2g+p$, the assertion (3) follows from Corollary \ref{cor:2g+p}. 
\end{proof}

The following is a reformulation of the above theorem in the light of \cite[Conjecture 3.7]{Lazarsfeld:ProjNormCurve}. Our result is a refinement of Green--Lazarsfeld's original expectation that if $\deg L \geq 2g+p+1$, then $K_{p,1}(C, \omega_C; L) \neq 0$ if and only if $\omega_C$ is not $p$-very ample (see \cite[page 87]{Lazarsfeld:ProjNormCurve}).

\begin{corollary}
Let $C$ be a smooth projective curve of genus $g \geq 2$, and  $L$ be a line bundle on $C$ with $\deg L \geq 2g+p$. Then $K_{p,1}(C, \omega_C; L) \neq 0$ if and only if one of the following holds:
\begin{enumerate}
	\item [(1)] $\omega_C$ is not $p$-very ample.
	\item [(2)] $C \subseteq  \nP H^0(C, H) = \nP^2$ is a plane curve of degree $p+3 \geq 4$ and $L=\omega_C \otimes H$. In this case, $\omega_C$ is $p$-very ample and $\deg L = 2g+p+1$.
	\item [(3)] $C$ is arbitrary and $L = \omega_C(\xi)$ for some $\xi \in C_{p+2}$ with $\dim |\xi| = 1$. In this case, $\omega_C$ is $p$-very ample and $\deg L = 2g+p$.
\end{enumerate}
\end{corollary}

\begin{proof}
If $\omega_C$ is $p$-very ample, then the corollary follows from Theorem \ref{main:01}. If $\omega_C$ is not $p$-very ample, then Corollary \ref{cor:K_{p+1,1}} shows that $K_{p, 1}(C, \omega_C; L) \neq 0$
\end{proof}

Now, we prove Theorem \ref{thm:effgon}.

\begin{proof}[Proof of Theorem \ref{thm:effgon}]
By the duality theorem (\cite[Theorem 2.c.1]{Green:KoszulI}), it is enough to show that if $\deg L \geq 2g+\gon(C)-2$, then 
$$
\text{$K_{p,1}(C, \omega_C; L) = 0$ for $0 \leq p \leq \gon(C)-2$ and $K_{\gon(C)-1,1}(C, \omega_C; L) \neq 0$}
$$ 
except for the cases (1) and (2) in which 
$$
\text{$K_{p,1}(C, \omega_C; L) = 0$ for $0 \leq p \leq \gon(C)-3$ and $K_{\gon(C)-2, 1}(C, \omega_C; L) \neq 0$}
$$
under the assumption that $K_{\gon(C)-1,1}(C, \omega_C; L) \neq K_{0,1}(C; L)^{\vee}$. 
If $K_{\gon(C)-1,1}(C, \omega_C; L) = K_{0,1}(C; L)^{\vee}=0$, then $g=2$ $(\gon(C)=2)$ and $\deg L = 4$ (cf. Corollary \ref{cor:2g+p} $(1)$). In this case, we only need to show that $K_{0,1}(C, \omega_C; L) \neq 0$ if and only if $L=\omega_C^2$. But this is a special case of Theorem \ref{main:01}. Now, assume that $K_{\gon(C)-1,1}(C, \omega_C; L) \neq K_{0,1}(C; L)^{\vee}$. For $0 \leq p \leq \gon(C)-3$, as $\deg L \geq 2g+p+1$ and $p \neq \gon(C)-2$, we get $K_{p,1}(C, \omega_C; L) = 0$ by Theorem \ref{main:01}. For $p=\gon(C)-2$, note that $\deg L \geq 2g+p$ and $\omega_C$ is $p$-very ample but not $(p+1)$-very ample. Then Corollary \ref{cor:2g+p} implies $K_{p+1, 1}(C, \omega_C; L) \neq 0$, and Theorem \ref{main:01} determines exactly when $K_{p,1}(C, \omega_C; L) =0$.
\end{proof}

We conclude this section by deriving a degree bound for $L$ proposed by Rathmann from Theorem \ref{thm:effgon}. In his journey of exploring an optimal bound for the gonality conjecture, Rathmann believes that the most natural degree bound would look like as $2g+g/2 + \epsilon$ together with some exceptional cases. It turns out that his intuition is indeed the case, and we confirm it in the following corollary, which also quickly implies Corollary \ref{main:02}.

\begin{corollary}\label{cor:2g+(g-2)/2}
	Let $C$ be a smooth projective curve of genus $g\geq 2$, and $L$ be a line bundle on $C$.
	If $\deg L \geq 2g+\lfloor (g-1)/2 \rfloor$, then 
	$$
	K_{p,1}(C; L) \neq 0 ~~\Longleftrightarrow~~ 1 \leq p \leq r(L)-\gon(C)
	$$
	except for the following two cases:
	\begin{enumerate}
		\item [(1)] $C \subseteq  \nP H^0(C, H) = \nP^2$ is a plane curve of degree $4,5,6$, and $L=\omega_C \otimes H$. $($In this case, $(g,\gon(C),\deg L) = (3,3,8), (6,4,15), (10, 5, 24)$, respectively.$)$ 
		\item [(2)] $C$ has maximal gonality, i.e., $\gon(C) = \lfloor (g+3)/2 \rfloor$, and $L=\omega_C(\xi)$ for an effective divisor $\xi$ of degree $\gon(C)$ with $\dim |\xi|=1$. 
	\end{enumerate}	
	In the exceptional cases, we have
	$K_{p,1}(C; L) \neq 0 ~~\Longleftrightarrow~~ 1 \leq p \leq r(L)-\gon(C)+1$.
\end{corollary}

\begin{proof}
	As $\gon(C) \leq \lfloor (g+3)/2 \rfloor$, we have $\deg L \geq 2g+\lfloor (g-1)/2 \rfloor \geq 2g+\gon(C)-2$. By Theorem \ref{thm:effgon}, we only need to show the following:
	\begin{enumerate}
		\item [(1)] If $C \subseteq  \nP H^0(C, H) = \nP^2$ is a plane curve of degree $d \geq 4$ and $L=\omega_C \otimes H$, then $d=4,5,6$.
		\item [(2)] If $L=\omega_C(\xi)$ for an effective divisor $\xi$ of degree $\gon(C)$ with $\dim |\xi|=1$, then $\gon(C)=\lfloor (g+3)/2 \rfloor$. 
	\end{enumerate}
	Suppose that we are in Case (1). We have $g=(d-1)(d-2)/2$ and $\gon(C)=d-1$. Since $\deg L = 2g+\gon(C)-1 \geq 2g + \lfloor (g-1)/2 \rfloor$, it follows that $d-2 \geq \lfloor (d^2-3d)/4 \rfloor$. This implies that $d=4,5,6$. Suppose that we are in Case (2). As $\deg L = 2g+\gon(C)-2 \geq 2g + \lfloor (g-1)/2 \rfloor$, we get $\gon(C) \geq \lfloor (g+3)/2 \rfloor$ and hence  $\gon(C)=\lfloor (g+3)/2 \rfloor$. 
\end{proof}

\begin{proof}[Proof of Corollary \ref{main:02}]
Note that $g \geq \lfloor (g+3)/2 \rfloor$. Thus $\deg L \geq 3g-2 \geq 2g+\lfloor (g-1)/2 \rfloor$. Suppose that we are in the exceptional case (1) of Corollary \ref{cor:2g+(g-2)/2}. Then $C \subseteq \nP H^0(C, \omega_C)$ is a plane quartic curve ($g=3$ and $L=\omega_C$), and  $L=\omega_C^2$. Suppose that we are in the exceptional case (2) of Corollary \ref{cor:2g+(g-2)/2}. Then $\deg L = 2g+\lfloor (g-1)/2 \rfloor \geq 3g-2$, and hence, $\lfloor (g+3)/2 \rfloor = g$. This yields $g=3$ or $2$. If $g=3$ ($\gon(C)=3$), then $C \subseteq \nP H^0(C, \omega_C) = \nP^2$ is a plane quartic curve. In this case, for any $\xi \in C_3$ with $\dim |\xi|=1$, there is $x \in C$ such that $\sO_C(\xi) = \omega_C(-x)$. Thus $L = \omega_C^2(-x)$. If $g=2$ ($\gon(C)=2$), then $\omega_C = \omega_C(\xi)$ for any $\xi \in C_2$ with $\dim |\xi|=1$. Thus $L=\omega_C^2$. Now, the corollary follows from Corollary \ref{cor:2g+(g-2)/2}.
\end{proof}

\bibliographystyle{alpha}

%\bibliography{../Reference/wref}

\begin{thebibliography}{[19]}

	\bibitem{Agostini:AsymSyz}
	Daniele Agostini, 
	\newblock {\em Asymptotic Syzygies and Higher Order Embeddings}, 
	\newblock Int. Math. Res. Not., \textbf{2022} (4): 2934--2967, 2022.
	
	\bibitem{Agostini:SecnatConj}
	Daniele Agostini,
	\newblock {\em The Martens--Mumford theorem and the Green--Lazarsfeld secant
	conjecture},
	\newblock to appear in J. Algebraic Geom.
	
	\bibitem{Aprodu:GonConj}
	Marian Aprodu,
	\newblock {\em Green–Lazarsfeld gonality conjecture for a generic curve of odd genus},
	\newblock Int. Math. Res. Not., \textbf{2004} (63): 3409--3416, 2004.
	
	\bibitem{AV}
	Marian Aprodu and Claire Voisin,
	\newblock {\em Green--Lazarsfeld’s conjecture for generic curves of large gonality},
	\newblock C. R. Math. Acad. Sci. Paris, \textbf{336} (4): 335--339, 2003.
	
	\bibitem{Butler}
	David C. Butler,
	\newblock {\em Global sections and tensor products of line bundles over a curve},
	\newblock Math. Z, \textbf{231}: 397--407, 1999.
	
	\bibitem{Castryck:LowerBoundGon}
	Wouter Castryck,
	\newblock {\em A lower bound for the gonality conjecture},
	\newblock Mathematika, \textbf{63} (2): 561--563, 2017.
	
	\bibitem{CKP} 
	Junho Choe, Sijong Kwak, and Jinhyung Park,
	\newblock {\em Syzygies of secant varieties of smooth projective curves and gonality sequences},
	\newblock preprint (2023), arXiv:2307.03405.
	
	\bibitem{CM}
	Marc Coppens and Gerriet Martnes,
	\newblock {\em Secant spaces and Clifford's theorem},
	\newblock Compositio Math., \textbf{78} (2): 193--212, 1991.
	
	\bibitem{DNP}
	Alexander~S. Duncan, Wenbo Niu, and Jinhyung Park,
	\newblock {\em A note on an effective bound for the gonality conjecture},
	\newblock preprint (2023), arXiv:2310.11419
	
	\bibitem{Ein:Gonality}
	Lawrence Ein and Robert Lazarsfeld,
	\newblock {\em The gonality conjecture on syzygies of algebraic curves of large
	degree},
	\newblock Publ. Math. Inst. Hautes \'Etudes Sci., \textbf{122}: 301--313, 2015.
	
	\bibitem{ENP}
	Lawrence Ein, Wenbo Niu, and Jinhyung Park,
	\newblock {\em Singularities and syzygies of secant varieties of nonsingular
	projective curves},
	\newblock Invent. Math., \textbf{222} (2): 615--665, 2020.
	
	\bibitem{Farkas:LinSyzGon}
	Gavril Farkas and Michael Kemeny,
	\newblock {\em Linear syzygies of curves with prescribed gonality},
	\newblock Adv. Math., \textbf{356}: 106810, 39, 2019.

	\bibitem{Green:KoszulI}
	Mark~L. Green,
	\newblock {\em Koszul cohomology and the geometry of projective varieties},
	\newblock J. Differential Geom., \textbf{19} (1): 125--171, 1984.
	
	\bibitem{Green:Koszul2}
	Mark~L. Green,
	\newblock {\em Koszul cohomology and the geometry of projective varieties} II, 
	\newblock J. Differential Geom., \textbf{20} (1): 279--289, 1984.
	
		
	\bibitem{Lazarsfeld:ProjNormCurve}
	Mark Green and Robert Lazarsfeld,
	\newblock {\em On the projective normality of complete linear series on an algebraic
	curve},
	\newblock Invent. Math., \textbf{83} (1): 73--90, 1986.
	
		
	\bibitem{Lazarsfeld:SomeResults}
	Mark Green and Robert Lazarsfeld,
	\newblock {\em Some results on the syzygies of finite sets and algebraic curves},
	\newblock Compositio Math., \textbf{67} (3): 301--314, 1988.
	
	\bibitem{Pareschi}
	Giuseppe Pareschi,
	\newblock {\em Gaussian maps and multiplication maps on certain projective varieties},
	\newblock Compositio Math., \textbf{98} (3): 219--268, 1995.
	
	\bibitem{Rathmann:EffBd}
	Juergen Rathmann,
	\newblock {\em An effective bound for the gonality conjecture},
	\newblock preprint (2016), {\em arXiv:1604.06072}.
	
	\bibitem{Voisin:GreenEven}
	Claire Voisin,
	\newblock {\em Green's generic syzygy conjecture for curves of even genus lying on a
	{$K3$} surface},
	\newblock J. Eur. Math. Soc. (JEMS), \textbf{4} (4): 363--404, 2002.
	
	\bibitem{Voisin:GreenOdd}
	Claire Voisin,
	\newblock {\em Green's canonical syzygy conjecture for generic curves of odd genus},
	\newblock Compos. Math., \textbf{141} (5): 1163--1190, 2005.
	
	
	\bibitem{Teixidor}
	Montserrat Teixidor i Bigas,
	\newblock {\em Syzygies using vector bundles},
	\newblock Trans. Amer. Math. Soc., \textbf{359} (2): 897--908, 2007.
	
\end{thebibliography}

\end{document}